\documentclass[11pt]{amsart}

\usepackage{geometry} 
\usepackage{verbatim} 
\usepackage{color,array}
\usepackage{amsmath}
\usepackage{amsfonts, amssymb, amsthm, setspace, latexsym}
\usepackage{multicol}
\usepackage{graphicx}
\usepackage[all,cmtip,curve, knot,frame]{xy} 
\usepackage{mathtools}

\numberwithin{equation}{section}
\newtheorem{theorem}{Theorem}[section]
\newtheorem{lemma}[theorem]{Lemma}
\newtheorem{proposition}[theorem]{Proposition}

\theoremstyle{definition}
\newtheorem{definition}[theorem]{Definition}
\newtheorem{remark}[theorem]{Remark}
\newtheorem{def/prop}[theorem]{Definition/Proposition}

\newcommand{\mf}{\mathfrak}
\newcommand{\cA}{\mathcal{A}}
\newcommand{\cM}{\mathcal{M}}
\newcommand{\cO}{\mathcal{O}}
\newcommand{\ZZ}{\mathbb{Z}}
\newcommand{\cD}{\mathcal{D}}
\newcommand{\HH}{\mathbb{H}_{t,q}}
\renewcommand{\H}{\operatorname{H}} 
\newcommand{\dS}{\Big/\!\!\Big/} 
\newcommand{\CC}{\mathbb{C}}
\newcommand{\tr}{\mathrm{tr}}

\newcommand{\rightloop}{%
           \mathrel{\raisebox{.1em}{%
           \reflectbox{\rotatebox[origin=c]{-90}{$\circlearrowright$}}}}}
\renewcommand{\d}{\partial}

\title{The Harish-Chandra isomorphism for quantum $GL_2$} 
\author{Martina Balagovi\'c and David Jordan}

\begin{document}
\begin{abstract}
We construct an explicit Harish-Chandra isomorphism, from the quantum Hamiltonian reduction of the algebra $\cD_q(GL_2)$ of quantum differential operators on $GL_2$, to the spherical double affine Hecke algebra associated to $GL_2$.  The isomorphism holds for all deformation parameters $q\in \CC^\times$ and $t\neq \pm i$, such that $q$ is not a non-trivial root of unity.  We also discuss its extension to this case.
\end{abstract}

\maketitle
\tableofcontents
\section{Introduction}

A fundamental construction in the geometric representation theory of a reductive algebraic group $G$ is Harish-Chandra's restriction homomorphism: given a conjugation-invariant differential operator on the Lie algebra, $\mathfrak{g}$, of $G$, we consider its restriction to the Cartan subalgebra $\mathfrak{h}\subset \mathfrak{g}$.  The restricted differential operator is invariant for the Weyl group $W$, but it may develop poles along hyperplanes stabilized by reflections in $W$.  Conjugating by the discriminant eliminates these poles, and produces a regular, $W$-equivariant differential operator on $\mathfrak{h}$.  This procedure leads to a homomorphism of algebras,
$$HC: D(\mathfrak{g})^G \to D(\mf{h})^W,$$
from the algebra $D(\mf{g})^G$ of ad-invariant differential operators on $\mf{g}$, to the algebra $D(\mf{h})^W$ of $W$-invariant differential operators on the Cartan subalgebra $\mf{h}$.  

Levasseur and Stafford's theorem \cite{LS} states that the restriction homomorphism descends to an isomorphism, called the Harish-Chandra isomorphism,
$$HC: D(\mathfrak{g})\!\!\!\underset{ker(\epsilon)}{\dS}\!\!\! G \xrightarrow{\cong} D(\mf{h})^W.$$
Here, $\epsilon:U\mathfrak{g}\to \CC$ denotes the ``co-unit" homomorphism, sending the generating subspace $\mathfrak{g}\subset U\mathfrak{g}$ to zero. More generally, for any ad-equivariant two-sided ideal $\mathcal{I}\subset U\mf g$, we denote by
$$D(\mathfrak{g})\underset{\mathcal{I}}{\dS} G:=\left(D(\mathfrak{g})\Big/D(\mathfrak{g})\cdot\hat{\mu}^\#(\mathcal{I})\right)^G,$$
the quantum Hamiltonian reduction along the quantum moment map, $\hat{\mu}^\#: U\mf{g}\to D(\mathfrak{g})$, corresponding to the derivative of the adjoint action.  The word ``quantum" here refers to the fact that $D(\mathfrak{g})$ and $U\mathfrak{g}$ quantize the standard symplectic structures on $T^*\mathfrak{g}$ and $\mathfrak{g}^*$, respectively, and that $\hat{\mu}^\#$ quantizes a classical moment map $\mu: T^*\mathfrak{g}\to \mathfrak{g}^*$. 

Specializing to the case $G=GL_N, \mf{g}=\mf{gl}_N, \mf{h}=\mathbb{C}^N, W=S_N$, the algebra $D(\mf{h})^W$ admits a canonical deformation to the \emph{spherical rational Cherednik algebra}, denoted $e\H_c(G)e$, and depending on a parameter $c\in\mathbb{C}$.  In \cite{EG}, Etingof and Ginzburg  extended Levasseur-Stafford's theorem, by defining a deformed Harish-Chandra isomorphism,  
$$HC_c:  D(\mf g)\underset{\mathcal{I}_c}{\dS} G \xrightarrow{\cong} e\H_c(G)e.$$
Here $\mathcal{I}_c\subset U\mf g$ is a certain ad-invariant, two-sided ideal quantizing the orbit $o_c$ of traceless matrices $X\in \mathfrak{sl}_N$ such that $X+c\cdot \operatorname{Id}$ has rank at most one.  The classical Hamiltonian reduction of $T^*\mathfrak{gl}_N$ along $o_c$ gives rise to the celebrated Calogero-Moser variety $\mathfrak{CM}_N$.

Each of the ingredients of the deformed Harish-Chandra homomorphism $HC_c$ admits a multiplicative deformation: 
\begin{itemize}
\item The universal enveloping algebra $U\mathfrak{g}$ $q$-deforms to a subalgebra $\cO_q(G)$ of the quantized universal enveloping algebra $U_q\mathfrak{g}$.  The algebra $\cO_q(G)$ is, moreover, a quantization of $G$ with its Semenov-Tian-Shansky Poisson bracket.  On the  other hand, the spherical rational Cherednik algebra $e\H_c(G)e$ $q$-deforms to the spherical double affine Hecke algebra, denoted $e\HH e$, which in turn quantizes the Ruijsenaars-Schneider integrable system $\mathfrak{RS}_N$.  We can encode these basic inputs and  outputs in the following ``diamonds of degenerations":
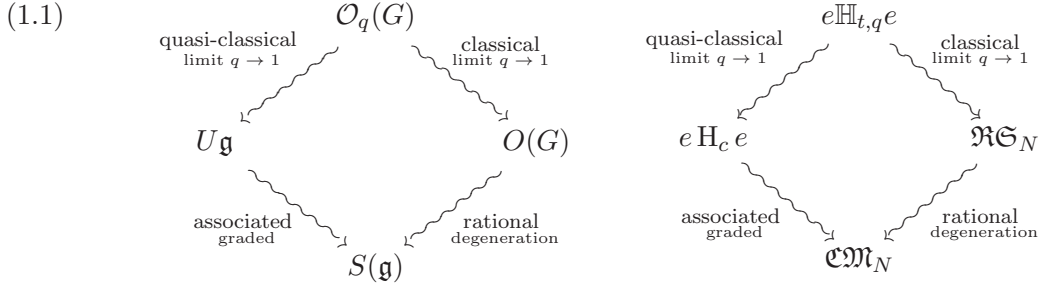
\begin{equation}\label{Uqdiag}\xymatrix@!C@R=30pt{ & \cO_q(G) \ar@{~>}[dl]_{\underset{\textrm{limit $q\to 1$}}{\textrm{\scriptsize quasi-classical }}} \ar@{~>}[dr]^{\underset{\textrm{limit $q\to 1$}}{\textrm{\scriptsize classical}}} & \\
                     U\mathfrak{g} \ar@{~>}[dr]_{\underset{\textrm{graded}}{\textrm{\scriptsize associated }}} && O(G) \ar@{~>}[dl]^{\underset{\textrm{degeneration}}{\textrm{\scriptsize rational }}}\\\
                     &  S(\mathfrak{g}) & } \qquad \xymatrix@!C@C=20pt@R=30pt{ & e\HH e \ar@{~>}[dl]_{\underset{\textrm{limit $q\to 1$}}{\textrm{\scriptsize quasi-classical }}} \ar@{~>}[dr]^{\underset{\textrm{limit $q\to 1$}}{\textrm{\scriptsize classical}}} & \\
                     e\H_ce \ar@{~>}[dr]_{\underset{\textrm{graded}}{\textrm{\scriptsize associated }}} && \mathfrak{RS}_N \ar@{~>}[dl]^{\underset{\textrm{degeneration}}{\textrm{\scriptsize rational }}}\\\
                     &  \mathfrak{CM}_N & }
                     \end{equation}
\end{itemize}

Each of the elements of Hamiltonian reduction also $q$-deform:  these basic inputs to Hamiltonian reduction are organized into the left hand side of Figure \ref{prism-diag}.  We have:
\begin{itemize}
\item   The algebra $\cD(\mathfrak{g})$ $q$-deforms to the algebra $\cD_q(G)$, of quantum differential operators on $G$.  The algebra $\cD_q(G)$ is, moreover, a quantization of the ``multiplicative cotangent bundle", $G\times G$.
\item The map $\hat{\mu}^\#$.  This is itself a $q$-deformation of a multiplicative moment map $\widetilde{\mu}:G\times G\to G$.
\item Each ideal $\mathcal{I}_c\subset U\mathfrak{g}$ $q$-deforms to an ad-invariant, two-sided ideal $\mathcal{I}_t\subset U_q\mathfrak{g}$, for $t\in\CC^\times$; each $\mathcal{I}_t$, in turn, quantizes an orbit $O_t$, consisting of special linear matrices $X\in SL_N$ such that $X+t\cdot\operatorname{Id}$ is rank one.
\item The quantum Hamiltonian reduction $\cD_q(G)\dS_{\!\!\!\mathcal{I}_t}U_q\mathfrak{g}$ is constructed analogously, both to the quantum Hamiltonian reduction of $\cD(\mathfrak{g})$, and to the multiplicative Hamiltonian reduction of $G\times G$.
\end{itemize}

It is thus a natural and important problem to construct the $q$-analog of the Harish-Chandra isomorphism,
\begin{equation*}HC_{q,t}: \cD_q(G)\dS_{\!\!\mathcal{I}_t} \!\!U_q\mathfrak{g} \xrightarrow{\cong} e\HH e,\end{equation*}
which would give rise to the degenerations on the right hand-side of Figure \ref{prism-diag}:

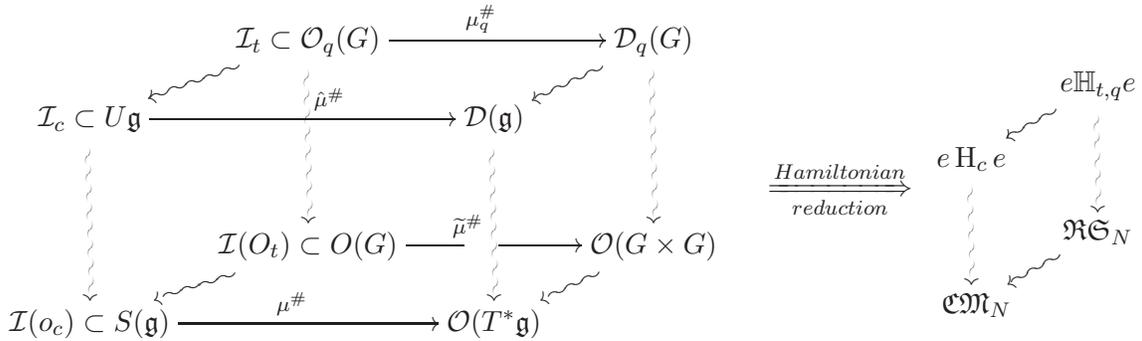
\begin{figure}[h]
\begin{equation*}
\label{muqdiag}
\begin{tabular}{m{2.5in} m{.5in} m{1.75in}}
\xymatrix@=12pt{ & \mathcal{I}_t\subset \cO_q(G) \ar@{~>}[dl]\ar@{~>}[ddd] \ar@{->}[rr]^{\mu^\#_q}&&\mathcal{D}_q(G) \ar@{~>}[dl] \ar@{~>}[ddd]&\\
                    \mathcal{I}_c\subset U\mathfrak{g} \ar@{~>}[ddd] \ar@{->}[rr]^{\phantom{===}\,\,\,\,\,\,\hat\mu^\#}&& \cD(\mathfrak{g}) \ar@{~>}[ddd]&&\\
                     &&&&\\
                     &\mathcal{I}(O_t) \subset O(G) \ar@{~>}[dl] \ar@{->}[rr]^{\!\!\!\!\!\!\!\widetilde{\mu}^\#}|{\phantom{==}} &&\cO(G\times G) \ar@{~>}[dl]\\
                      \mathcal{I}(o_c) \subset S(\mathfrak{g})\ar@{->}[rr]^{\mu^\#} &&\cO(T^*\mathfrak{g})  &} & $\!\!\!\!\!\!\!\xRightarrow[reduction]{Hamiltonian}$ &
\xymatrix@=12pt{&&\\ & e\HH e \ar@{~>}[dl] \ar@{~>}[dd] & \\
                     e\H_ce \ar@{~>}[dd] &\\
& \mathfrak{RS}_N \ar@{~>}[dl]&\\\
                      \mathfrak{CM}_N && }
\end{tabular}
\end{equation*}
\caption{The elements of classical, multiplicative, quantum, and quantum multiplicative Hamiltonian reduction, and their degenerations.}\label{prism-diag}
\end{figure}

Such quantum Harish-Chandra isomorphisms have been constructed for special parameters, first in \cite{VV}, when $q$ is a root of unity, via quantum differential operators on $GL_N\times \mathbb{P}^{N-1}$, and subsequently in \cite{J}, in the formal neighborhood of $q=1$, as a special case of a quantum multiplicative quiver variety.

Each of these approaches rely on being ``close to" the classical setting:  in \cite{VV} Azumaya algebras are employed to reduce to a geometric problem on the spectrum of the center, while \cite{J} relies upon formal deformation arguments.  Neither method can be applied to arbitrary -- or even generic -- values of the parameter $q\in\mathbb{C}$.

\subsection{Main result}
This paper is devoted to constructing a $q$-deformed Harish-Chandra isomorphism ``far from" the classical setting -- when $q$ is not a root of unity -- but only in the case $N=2$.  Our main result is Theorem \ref{main-result}, in which we construct an explicit isomorphism of algebras:
$$HC_{q,t}: \cD_q(GL_2) \underset{\mathcal{I}_t}{\dS} U_q\mathfrak{gl}_2 \xrightarrow{\cong} e\HH(GL_2) e,$$
and which holds for any $q\in\mathbb{C}^\times$ not a root of unity, and for all $t\neq\pm i$.

Our construction of $HC_{q,t}$ consists of giving explicit presentations and PBW bases of both algebras (see Theorems \ref{thm-sDAHApresentation} and \ref{HamiltonPresentation}), and directly defining and computing $HC_{q,t}$ against these PBW bases.  To this end, we exploit certain natural $\ZZ^2$-gradings on each algebra; while the graded pieces are infinite-dimensional, we express each algebra as an Ore localization of a certain auxiliary algebra with finite-dimensional graded pieces, and show that $HC_{q,t}$ is a graded isomorphism there, compatibly with the localization.  

\begin{remark} In light of \cite{VV}, \cite{J}, and the present paper, it is natural to conjecture the existence of a $q$-deformed Harish-Chandra isomorphism in complete generality.  In a forthcoming paper \cite{JW}, certain explicit formulas are given for elements $c_1,\ldots, c_N$ which $q$-deform the coefficients of the characteristic polynomial, thus generalizing $\tr_q$ and $\det_q$ of Lemma \ref{trdetarecenter}; these give natural generators for subalgebra of invariants in $\cO_q(G)$. The formulas in Theorem \ref{main-result} suggest that such a $q$-deformed Harish-Chandra homomorphism might be defined simply by sending the collection of generating invariants, $c_i$, in each $\cO_q(G)$-subalgebra of $\cD_q(G)$, to the elementary $q$-symmetric polynomials $P_i$ and $Q_i$, respectively.  

We note, however, that none of three approaches from \cite{VV}, \cite{J}, and the present paper suggest how to prove this gives a well-defined homomorphism, let alone that it is an isomorphism.  We expect the elementary methods in this paper could be extended to $N=3,4,\ldots$ but are unlikely to lead to proof for arbitrary $N$, owing to the increasing complexity in presentation of the spherical DAHA in high rank.
\end{remark}

\begin{remark}[Roots of unity]
When the quantum parameter $q$ is a root of unity, there is an interesting phenomenon, related to the quantum Frobenius Hamiltonian reduction of \cite{G}:  the algebra $\cD_q(G)$ develops a large center, and many more vectors become invariant for the quantum group action.  Meanwhile, the functor of taking invariants is no longer exact, so the simplification of Lemma \ref{ReduceToInvariants} cannot be made.  It is interesting to note that, nevertheless, the isomorphism $HC_{q,t}$ is perfectly well-defined when $q$ is a root of unity, only it identifies the spherical DAHA at a root of unity with an abstract algebra which cannot in turn be identified with the quantum Hamiltonian reduction.  This suggests that, when $q$ is a root of unity, the notion of Hamiltonian reduction employed here should be replaced by one involving divided powers quantum groups, so that the isomorphism $HC_{q,t}$ recovers its natural interpretation.
\end{remark}

\subsection{Relation to factorization homology} 
Let us briefly explain the connection of the present work to \cite{BJ, BZBJ1} and to \cite{BZBJ2}, where Theorem \ref{main-result} has already been announced.  Since this relation is only motivational in the present paper, we will be informal, refering to \cite{BZBJ1},\cite{BZBJ2} for complete details.  Let us fix the braided tensor category $\cA=\operatorname{Rep}_q G$ of integrable  modules for the quantum group, its factorization homology defines an invariant of surfaces valued in ``pointed" (a.k.a. $E_0$-) categories, i.e. categories endowed with a distinguished object:
\begin{align*}\textrm{Surfaces} &\longrightarrow \textrm{Pointed categories},\\
S\quad&\mapsto \quad \left(\int_S\cA,\,\, \cO_S \right).\end{align*}
The algebra $\cO_q(G)$ arises naturally in this context:  its representation theory, the category $\cO_q(G)$-mod$_\cA$ of $\cO_q(G)$-modules in $\cA$, computes the factorization homology $\int_{Ann}\cA$, of $\cA$ on the annulus, and we have $\cO_{Ann} \cong \cO_q(G)$ itself.

Given a $U_q\mathfrak{g}$-equivariant two-sided ideal of $\cO_q(G)$, such as $\mathcal{I}_t$, the category $\cM$ of modules for the algebra $\cO_q(G)/\mathcal{I}_t$ therefore defines a ``braided $\cA$-module category".  Braided $\cA$-module categories serve as the point defects in the associated topological field theory: the factorization homology of the pair $(U_q\mathfrak{g}, \mathcal{I}_t)$ defines an invariant of surfaces with marked points, again valued in pointed categories:
\begin{align*}
\textrm{Marked surfaces} &\longrightarrow\quad \textrm{Pointed categories},\\
(S,x)\,\, &\mapsto\,\, \left(\int_{(S,x)} \!\!\!\!\!\!(\cA,\cM),\,\, \cO_{(S,x)} \right).
\end{align*}

A theorem in \cite{BZBJ2} identifies the endomorphism algebra of the distinguished object for a torus $T^2$, marked at a single point with the braided module category determined by $\mathcal{I}_t$, with the quantum Hamiltonian reduction algebra $\cD_q(G)\dS_{\!\!\mathcal{I}_t}U_q\mathfrak{gl}_2$.

Hence the results of this paper identify this endomorphism algebra, in the case $G=GL_2$, with the spherical double affine Hecke algebra $e\HH e$.  This points to a new construction of the representation theory of the DAHA in topological terms. The factorization homology studied in \cite{BZBJ1},\cite{BZBJ2} is expected to extend to 3-dimensional TFT: its values on knot complements therefore give rise, functorially, to representations of $e\HH e$.

Hence Theorem \ref{main-result} should provide a construction via factorization homology of several expected topological invariants valued in the representation theory of the DAHA, such as the quantum A- \cite{GS, FGL} and DAHA-Jones \cite{C3} polynomials associated to knots in $S^3$, and the DAHA modules of \cite{GORS,ORS} associated to plane singularities.  Our results can also be compared to \cite{BS}, who study appearances of the spherical DAHA as modified skein algebras: in particular, at $t=1$, both constructions recover \cite{FG}, and both constructions are aimed at restoring the lost parameter $t$.  We refer to the reader to the papers \cite{BZBJ1,BZBJ2} for a more detailed discussion.

\subsection{Outline}
An outline of the paper is as follows. In Section \ref{review-DAHA}, we recall the definition of the double affine Hecke algebra $\HH$ of $GL_2$ and its spherical subalgebra $e\HH e$, and we give an explicit presentation and a PBW basis for the latter (Theorem \ref{thm-sDAHApresentation}). In Section \ref{review-cDq} we recall the definition of $\cD_q(GL_2)$ and its Hamiltonian reduction $\cD_q(GL_2) \underset{\mathcal{I}_t}{\dS} U_q\mathfrak{g}$, and we give an explicit presentation and a PBW basis for the latter (Theorem \ref{HamiltonPresentation}). Section \ref{mainstatement} contains the main result, Theorem \ref{main-result}, which constructs an isomorphism between the spherical DAHA and the Hamiltonian reduction of $\cD_q(GL_2)$.  This result follows directly from the presentations of these two algebras from Theorems \ref{thm-sDAHApresentation} and \ref{HamiltonPresentation}, which constitute the main technical contribution of the paper. The proofs of Theorems \ref{thm-sDAHApresentation} and \ref{HamiltonPresentation} are delayed until Section \ref{sec-DAHAPResentation}, and Section \ref{PBW-DqG}, respectively.

\subsection{Acknowledgments} The authors would like to thank Pavel Etingof for several helpful conversations over the years.  The first author was supported by the EPSRC grant EP/K025384/1. The research of the second author is supported by ERC Starting Grant 637618.




\section{Double affine Hecke algebra $\mathbb{H}_{t,q}$ of $GL_2$ and its spherical subalgebra}\label{review-DAHA}

\subsection{Parameters}
Let us fix as a base ring $R_{t,q}=\mathbb{C}[t^{\pm1},q^{\pm1}][(t^2+1)^{-1}]$, 
the commutative ring of rational functions in variables $t$ and $q$ whose denominators are products of powers of $t,q$ and $t^2+1$.

For any numbers $q',t' \in \mathbb{C}^\times$ with $(t')^2+1\ne 0$, the quotient of $R_{t,q}$ by the corresponding maximal ideal $\mathfrak{m}=\langle(q-q'),(t-t')\rangle$ gives the evaluation map $R_{t,q}\twoheadrightarrow R_{t,q}/\mathfrak{m}\cong \mathbb{C}$.  For any algebra $\mathbf{A}$ over $R_{t,q}$, the evaluation map gives its specialization, an algebra $\mathbf{A}/\mathfrak{m}\mathbf{A}$ over over $\mathbb{C}$.

For the rest of the paper we will work with algebras depending on parameters $q$ and $t$; we will treat the corresponding algebras over $R_{t,q}$ and their specializations simultaneously. Most arguments and results apply in both settings; we will point out the distinction only when necessary.  In particular, the condition ``$q$ is not a non-trivial root of unity" means either that we work over $R_{q,t}$ (so the condition is vacuously true), or over $\CC$, but with the stated restriction.

\subsection{The double affine Hecke algebra $\mathbb{H}_{t,q}$}

\begin{definition}\label{defDAHA}
The double affine Hecke algebra (DAHA) of type $GL_2$ is the associative algebra  $\mathbb{H}_{t,q}$ 
with generators 
$$T^{\pm 1}, X_{1}^{\pm 1}, X_{2}^{\pm 1}, Y_{1}^{\pm 1}, Y_{2}^{\pm 1}$$
and relations
\begin{align}
\begin{aligned}
TX_1T=X_2, \qquad T^{-1}Y_1T^{-1}=Y_2, \qquad X_1X_2 = X_2X_1,\qquad Y_1Y_2=Y_2Y_1,\\
Y_1X_1X_2=q^2X_1X_2Y_1, \qquad X_1^{-1}Y_2=Y_2X_1^{-1}T^{-2}, \qquad (T+t^{-1})(T-t)=0.
\end{aligned}\label{defDAHA-relations1}
\end{align}
\end{definition}

\begin{remark}
This definition first appeared in \cite{C1}. There are several different conventions in the literature for the presentation of DAHA and the choice of parameters. Definition \ref{defDAHA} follows the conventions of \cite{GN} (see also \cite{C1}, \cite{C2}, and \cite{O}) up to replacing the parameters $t$ by $t^2$ and $q$ by $q^2$. This rescaling is not significant, and it is dictated by the need to align the parameters for DAHA with the parameters for $\cD_q(GL_2)$ in the next section. 
\end{remark}

\begin{remark}Let $S_2$ be the symmetric group  of order $2$. The algebra $\mathbb{H}_{1,1}$ is isomorphic to 
the semidirect product
$\CC[S_2]\ltimes \mathbb{C}[X_1^{\pm1},X_2^{\pm1},Y_1^{\pm1},Y_2^{\pm1}]$, and the algebra $\mathbb{H}_{t,q}$ is a flat deformation of $\mathbb{H}_{1,1}$.
\end{remark}

The following proposition gives the reordering relations for any two generators $X_i,Y_j$ and the PBW basis. Similar relations can be written for any two generators $X_i^{\pm1},Y_j^{\pm1}$.

\begin{proposition}\label{prop-AllRelnsDAHA}
\begin{enumerate}
\item The following relations are also satisfied in $\mathbb{H}_{t,q}$:
\begin{align}
\begin{aligned}
T^2=(t-t^{-1})T+1, \quad &  \quad T^{-1}=T- (t-t^{-1}),
\end{aligned}\label{defDAHA-relations3}
\end{align}
\begin{align}
\begin{aligned}
X_1Y_1=q^{-2}T^{-2}Y_1X_1,\quad&\quad
X_1Y_2=Y_2X_1 +(t-t^{-1}) T^{-1}Y_1X_1,\\
X_2Y_1=Y_1X_2+(t-t^{-1})q^{-2}T^{-1}Y_1X_1  \quad&\quad 
X_2Y_2=q^{-2}Y_2X_2T^{-2}.
\end{aligned}\label{defDAHA-relations2}
\end{align}
\item The set
 $$\{T^\epsilon Y_1^{a_1} Y_2^{a_2}X_1^{b_1}X_2^{b_2} \ | \  \epsilon \in \{0,1\}, a_i,b_i \in \mathbb{Z}\}$$ forms a basis of  $\mathbb{H}_{t,q}$.
\end{enumerate}
\end{proposition}
\begin{proof}
Part (1) is proved by direct computations. Part (2) was proved in \cite{C2}. For illustration we prove the relations \eqref{defDAHA-relations2}. 

First, rewriting the relation $X_1^{-1}Y_2=Y_2X_1^{-1}T^{-2}$ from \eqref{defDAHA-relations1} as $Y_2T^2X_1=X_1Y_2$, we have
\begin{align}
X_1Y_2&\stackrel{\phantom{\eqref{defDAHA-relations3}}}{=}Y_2T^2X_1 \notag\\
&\stackrel{\eqref{defDAHA-relations3}}{=}Y_2X_1+(t-t^{-1})Y_2TX_1 \notag\\
&\stackrel{\eqref{defDAHA-relations1}}{=}Y_2X_1+(t-t^{-1})T^{-1}Y_1X_1. \label{defDAHA-relations2.1}
\end{align}
Conjugating the relation $X_1^{-1}Y_2=Y_2X_1^{-1}T^{-2}$  by $T^{-1}$ and rewriting, we have
\begin{align}
T^{-1}X_1^{-1}Y_2T&=T^{-1}Y_2X_1^{-1}T^{-1} \notag\\
X_2^{-1}Y_1&=T^{-2}Y_1 X_2^{-1}. \label{defDAHA-relations2.2}
\end{align}
Next, we have
\begin{align}
X_1Y_1&\stackrel{\phantom{\eqref{defDAHA-relations3}}}{=}X_2^{-1}X_1X_2Y_1 \notag\\
&\stackrel{\eqref{defDAHA-relations1}}{=}q^{-2}X_2^{-1}Y_1X_1X_2 \notag\\
&\stackrel{\eqref{defDAHA-relations2.2}}{=}q^{-2}T^{-2}Y_1 X_1\label{defDAHA-relations2.3}
\end{align}
Conjugating \eqref{defDAHA-relations2.3} by $T$, we have
\begin{align}
X_2Y_2&\stackrel{\eqref{defDAHA-relations1}}{=} TX_1Y_1T^{-1}\notag\\
&\stackrel{\eqref{defDAHA-relations2.3}}{=} q^{-2}T^{-1}Y_1 X_1 T^{-1}\notag\\
&\stackrel{\eqref{defDAHA-relations1}}{=} q^{-2}Y_2 X_2 T^{-2}.\label{defDAHA-relations2.4}
\end{align}
Finally, combining \eqref{defDAHA-relations2.1} and \eqref{defDAHA-relations2.3}, we get
\begin{align}
X_2Y_1&\stackrel{\eqref{defDAHA-relations1}}{=}TX_1T^2Y_2T \notag\\
&\stackrel{\eqref{defDAHA-relations3}}{=}TX_1Y_2T+(t-t^{-1})TX_1TY_2T \notag\\
&\stackrel{\eqref{defDAHA-relations2.1}}{=}TY_2X_1T+(t-t^{-1})Y_1X_1T+(t-t^{-1})TX_1TY_2T \notag\\
&\stackrel{\eqref{defDAHA-relations1}}{=}Y_1T^{-1}X_1T+(t-t^{-1})Y_1X_1T+(t-t^{-1})TX_1Y_1 \notag\\
&\stackrel{\eqref{defDAHA-relations3}}{=}Y_1TX_1T+(t-t^{-1})TX_1Y_1 \notag\\
&\stackrel{\eqref{defDAHA-relations2.3}}{=}Y_1X_2+(t-t^{-1})q^{-2}T^{-1}Y_1 X_1.
\label{defDAHA-relations2.5}
\end{align}

\end{proof}

The algebra $\mathbb{H}_{t,q}$ is $\mathbb{Z}^2$-graded by
$$\deg X_1= \deg X_2=(0,1),\quad \deg Y_1 = \deg Y_2 =(1,0),\quad \deg T=(0,0).$$ This grading is inner: it follows from \eqref{defDAHA-relations2} that any $h\in \mathbb{H}_{t,q}$ with $\deg{h}=(M,N)$ satisfies
\begin{align}
Y_1Y_2 \, h = q^{2N} h Y_1Y_2 \quad & \quad X_1X_2 \, h = q^{-2M} h X_1X_2. \label{X1X2-is-q-central}
\end{align}

\subsection{The spherical double affine Hecke algebra $e\mathbb{H}_{t,q}e$}

The subalgebra of the double affine Hecke algebra generated by $T$ and $T^{-1}$ is the (finite) Hecke algebra of the symmetric group $S_2$. For $t^2+1 \ne 0$, this algebra is isomorphic to the group algebra of $S_2$. Denoting the only nontrivial element of the symmetric group by $s$, the identification is given by 
$$T=\frac{t+t^{-1}}{2}s+\frac{t-t^{-1}}{2}.$$
The trivial idempotent of $S_2$ is
$$e=\frac{1+s}{2}=\frac{1+tT}{1+t^2} \in \mathbb{H}_{t,q}.$$
It satisfies 
\begin{align}
eT=Te=te, \qquad eT^{-1}=T^{-1}e=t^{-1}e. \label{eT=Te=te}  
\end{align}
This action of this idempotent on any representation is the projection to the trivial representation of $S_2$, or equivalently to the representation of the Hecke algebra in which $T$ acts by the eigenvalue $t$.

\begin{definition}
The spherical double affine Hecke algebra is the subalgebra $e\mathbb{H}_{t,q}e$ of $\mathbb{H}_{t,q}$.
\end{definition}

We will give an explicit presentation of the spherical DAHA by generators and relations, and the resulting PBW type basis for it. 

\begin{theorem}\label{thm-sDAHApresentation}
\begin{enumerate}
\item The spherical double affine Hecke algebra $e\mathbb{H}_{t,q}e$ is isomorphic to the algebra with generators $P_1,P_2^{\pm1},Q_1,Q_2^{\pm1},R$ and relations:

\vspace{-0.5cm}
\hspace{-1cm}
\begin{minipage}[t]{0.41\textwidth}
\begin{align}
P_2P_1&=P_1P_2 \label{P1P2} \\
Q_2Q_1&=Q_1Q_2  \label{Q1Q2} \\
P_2Q_2&=q^{-4}Q_2P_2 \label{P2Q2} \\
P_2Q_1&=q^{-2}Q_1P_2\label{P2Q1} \\
P_1Q_2&=q^{-2}Q_2P_1 \label{P1Q2} \\
P_1Q_1&= Q_1P_1+(q^{-2}-1)R \label{P1Q1} 
\end{align}
\end{minipage}
\begin{minipage}[t]{0.565\textwidth}
 \begin{align}
P_2R&= q^{-2}RP_2 \label{P2R}\\
RQ_2&= q^{-2}Q_2R \label{RQ2} \\
P_1R&=q^{-2}RP_1+(1-q^{-2})Q_1P_2  \label{P1R} \\
RQ_1&= q^{-2}Q_1R+(1-q^{-2})Q_2P_1 \label{RQ1} \\
R^2&= (1+t^2)(q^{-2}+t^{-2})Q_2P_2- \label{R2reln}  \\
& -q^{-2}Q_2P_1^2-q^{-2}Q_1^{2}P_2+q^{-2}Q_1RP_1. \notag
\end{align}
\end{minipage}

\vspace{0.6cm}
The isomorphism is given by 
\begin{align}
\begin{aligned}
\Phi(P_1)= e(X_1+X_2)e, \quad &  \Phi(P_2)=  e(X_1X_2)e, \quad &  \Phi(R)  = e(t^{-2}Y_1X_1+Y_2X_2)e,\\
\Phi(Q_1)=  e(Y_1+Y_2)e,\quad &   \Phi(Q_2)=  e(Y_1Y_2)e.&\\
\end{aligned}
\label{sDAHA-pres-iso}
\end{align}
\item After the identification given by the isomorphism $\Phi$, the set  $$\{Q_1^{a_1}Q_2^{a_2}R^\epsilon P_1^{b_1}P_2^{b_2} \ | \  \epsilon \in \{0,1\}, a_1,b_1 \in \mathbb{N}_0, a_2,b_2 \in \mathbb{Z} \}$$
forms a basis of the spherical double affine Hecke algebra $e\mathbb{H}_{t,q}e$. 
\end{enumerate}
\end{theorem}

The presentation of the spherical DAHA by generators and relations given in Theorem \ref{thm-sDAHApresentation} is one of the main technical steps of the paper and is used in the proof of the main result, Theorem \ref{main-result}. Section \ref{sec-DAHAPResentation} is dedicated to the proof of its proof.

\section{The algebra $\cD_q(GL_2)$ of quantum differential operators and its Hamiltonian reduction}\label{review-cDq}

In this section, we recall the definitions quantum enveloping algebra $U_q\mathfrak{gl}_2$, the quantum coordinate algebra $\cO_q(GL_2)$, the algebra $\cD_q(GL_2)$ of polynomial quantum differential operators, and the Hamiltonian reduction $\cD_q(GL_2)\dS_{\!\!\! \mathcal{I}_t} U$ of $\cD_q(GL_2)$.  We will follow the conventions of \cite{KS} and \cite{J}. 

\subsection{The quantum enveloping algebra $U=U_q\mathfrak{gl}_2$}

\begin{definition} The quantum enveloping algebra $U=U_q\mathfrak{gl}_2$ is the Hopf algebra with
\begin{itemize}
\item generators $$E,F,K_1^{\pm 1}, K_2^{\pm 1};$$
\item relations
\begin{align*}
K_1EK_1^{-1}=qE, \quad & \quad K_2EK_2^{-1}=q^{-1}E,\\ 
K_1FK_1^{-1}=q^{-1}F, \quad & \quad K_2FK_2^{-1}=qF,\\
K_1K_2=K_2K_1 , \quad & \quad EF-FE=\frac{K_1K_2^{-1}-K_1^{-1}K_2}{q-q^{-1}};
\end{align*}
\item the Hopf structure:
\begin{align*}
\Delta(E)&=E\otimes K_1K_2^{-1} + 1\otimes E,  & \quad S(E)&=-EK_1^{-1}K_2 ,  & \quad  \varepsilon(E)=0,\\
\Delta(F)&= F\otimes 1+K_1^{-1}K_2\otimes F,  & \quad S(F)&=-K_1K_2^{-1}F ,  & \quad \varepsilon(F)=0,\\
\Delta(K_i)&= K_i\otimes K_i,  & \quad S(K_i)&=K_i^{-1} ,  & \quad \varepsilon(K_i)=1.\\
\end{align*}
\end{itemize}
\end{definition}

The vector representation $V$ of $U$ has ordered basis $\{e_{-1}, e_{1}\}$, with action: 
\begin{align*}
\rho_V(K_1)=\left[\begin{array}{cc}q^{-1} & 0 \\ 0 & 1 \end{array}\right], \  \rho_V(K_2)=\left[\begin{array}{cc}1 & 0 \\ 0 & q^{-1} \end{array}\right], \ \rho_V(E)=\left[\begin{array}{cc}0 & 0 \\ 1 & 0 \end{array}\right], \ \rho_V(F)=\left[\begin{array}{cc}0 & 1 \\ 0 & 0 \end{array}\right].
\end{align*}

The Hopf algebra $U$ is quasi-triangular, with the R-matrix
\begin{align*}
\mathcal{R}= q^{{H\otimes H}} \, \cdot \, \sum_{n\ge 0}q^{\frac{n(n-1)}{2}} (q-q^{-1})^n \frac{1}{[n]_q!} E^n\otimes F^n. 
\end{align*}
Here $q^{{H\otimes H}}$ is the operator which acts by a scalar on the tensor product of any two weight spaces. More precisely, for $V_1,V_2$ representations of $U$, and $v_1 \in V_1,v_2 \in V_2$ vectors such that $K_i\rhd v_j=q^{\alpha_{ij}}v_j$, the action of $q^{{H\otimes H}}$ on their tensor product is given by $q^{{H\otimes H}}\rhd (v_1\otimes v_2)=q^{\alpha_{11}\alpha_{12}+\alpha_{21}\alpha_{22}}(v_1\otimes v_2)$. 

\subsection{A presentation of $U$ via L-matrices}
We recall an alternative presentation of $U$ 
 which uses the R-matrix $\mathcal{R}$, its inverse $\mathcal{R}^{-1}$ and the vector representation $\rho_V$ (see \cite{KS}). Define $\underline{R}$, $L^+$, $L^-$ by 
\begin{align*}
\underline{R}&=(\rho_V \otimes \rho_V)(\mathcal{R})\in \mathrm{Mat}_{2\times 2} \otimes \mathrm{Mat}_{2\times 2}\\
L^+&=(\mathrm{id}\otimes \rho_V)(\mathcal{R}) \in \mathrm{Mat}_{2\times 2}(U)\\ 
L^-&=( \rho_V \otimes \mathrm{id})(\mathcal{R}^{-1}) \in \mathrm{Mat}_{2\times 2}(U),
\end{align*}
where $\mathrm{Mat}_{2\times 2}$ denotes the set of $2\times 2$ matrices with entries in the ground ring ($R_{t,q}$ or $\mathbb{C}$), and $\mathrm{Mat}_{2\times 2}(U)$ denotes the set of $2\times 2$ matrices with entries in the algebra $U$. Let $L^\pm_1=L^\pm \otimes \mathrm{id}$, $L^\pm_2=\mathrm{id}\otimes L^\pm$.

\begin{proposition}[\cite{KS}, Theorem 8.33]The Hopf algebra $U$ is isomorphic to the Hopf algebra with:
\begin{itemize}
\item generators given by the entires of the matrices
\begin{align*}
L^+=\left[\begin{array}{cc}l^{+1}_1 & l^{+1}_2 \\ 0 & l^{+2}_2 \end{array}\right], \quad  L^-=\left[\begin{array}{cc}l^{-1}_1 &0 \\ l^{-2}_1 & l^{-2}_2 \end{array}\right];
\end{align*}
\item relations given by the entries of the matrix equations
\begin{align*}
L_1^\pm L_2^\pm \underline{R}=\underline{R}  L_2^\pm  L_1^\pm, \quad \quad L_1^- L_2^+ \underline{R}=\underline{R}  L_2^+  L_1^-, \quad \textrm{ and } \quad  & l^{-i}_i=(l^{i}_i)^{-1};
\end{align*}
\item which may be written explicitly as:

\begin{align*}
\begin{aligned}
&l^{+1}_1l^{-1}_1 = l^{+2}_2l^{-2}_2 = 1 & \quad &l^{\pm1}_1 l^{\pm 2}_2 = l^{\pm 2}_2 l^{\pm 1}_1& &\\
&l^{+1}_1l^{+1}_2 = q^{-1}l^{+1}_2l^{+1}_1& \quad & l^{+2}_2l^{+1}_2 = ql^{+1}_2l^{+2}_2& &\\
&l^{+1}_1l^{-2}_1 = ql^{-2}_1l^{+1}_1& \quad & l^{+2}_2l^{-2}_1 = q^{-1}l^{-2}_1l^{+2}_2& &\\
&l^{-2}_1l^{+1}_2-l^{+1}_2l^{-2}_1 = (q-q^{-1})(l^{+2}_2l^{-1}_1-l^{+1}_1l^{-2}_2);& \quad & & &
\end{aligned}
\end{align*}

\item the Hopf structure given by:
\begin{align*}
S(L^\pm)=(L^\pm)^{-1}, \quad  \quad \Delta(l^{\pm i}_j)=\sum_k l^{\pm i}_k\otimes l^{\pm k}_j, \quad \quad \varepsilon(l^{\pm i}_j)=\delta_{ij}.
\end{align*}
\end{itemize}

The isomorphism is given by:
\begin{align*}
l^{\pm i}_i=K_i^{\mp 1}, \quad l^{+1}_2=(q-q^{-1})K_1^{-1}E, \quad l^{-2}_1=-(q-q^{-1})FK_1.
\end{align*}
\end{proposition}

\subsection{The adjoint action and the quantum coordinate algebra $\cO_q(GL_2)$}\label{sec-adjoint-action-Oq}

The Hopf algebra $U$ 
acts on itself by the \emph{adjoint action}, defined as 
\begin{align*}
x \rhd y=x_{(1)}yS(x_{(2)}) \quad \textrm{ for } \quad \Delta(x)=\sum x_{(1)}\otimes x_{(2)}.
\end{align*}
The locally finite subspace with respect to this action is the subalgebra generated by the entries of the matrix
$L^+S(L^-)$, and by $(l^{+1}_1l^{+2}_2)^{-1}.$  This subspace has another interpretation, as the quantum algebra of functions on the group $GL_2$ (see Proposition \ref{locally-finite}:

\begin{definition} \label{def-O_q}
Let $\cO^+_q(GL_2)$ be the algebra with 
\begin{itemize} 
\item generators $l^i_j$, organized in a matrix $L=\left[\begin{array}{cc}\ell^1_1 & \ell^1_2 \\ \ell^2_1 & \ell^2_2 \end{array}\right]$; 
\item relations given by the entries of the matrix equation
$$R_{21}L_1RL_2=L_2R_{21}L_1R,$$
\item which may be written explicitly as:
\begin{align}
\begin{aligned}\label{relns-aa}
&\ell^1_2\ell^1_1 = \ell^1_1\ell^1_2 + (1-q^{-2})\ell^1_2\ell^2_2 & \quad  &\ell^2_2\ell^1_1 = \ell^1_1\ell^2_2& &\\
&\ell^2_1\ell^1_1 = \ell^1_1\ell^2_1 - (1-q^{-2})\ell^2_2\ell^2_1& \quad  &\ell^2_2\ell^1_2 = q^2\ell^1_2\ell^2_2& &\\
&\ell^2_1\ell^1_2 = \ell^1_2\ell^2_1 + (1-q^{-2})(\ell^1_1\ell^2_2 - \ell^2_2 \ell^2_2)& \quad  &\ell^2_2\ell^2_1 = q^{-2}\ell^2_1\ell^2_2\\
\end{aligned}
\end{align}

\end{itemize}
\end{definition}

\begin{proposition} The element ${\det}_q(L):=\ell^1_1\ell^2_2-q^2\ell^1_2\ell^2_1$ is central in $\cO^+_q(GL_2)$.\end{proposition}

\begin{definition} Let 
\begin{align*}
\cO_q(GL_2)&=\cO^+_q(GL_2)[{\det}_q(L)^{-1}]
\end{align*}
denote the algebra obtained from $\cO^+_q(GL_2)$ localizing at the central element ${\det}_q(L)$. We call $\cO_q(GL_2)$ the quantum coordinate algebra of $GL_2$. \footnote{In fact there are two candidates for this name, and for the notation $\cO_q(G)$, the other being the restricted dual to $U_q$, also known as the Fadeev-Reshetikhin-Takhtajan (FRT) algebra.  We will not make use of the FRT algebra in this paper.}
\end{definition}

\begin{proposition}[\cite{KS}]\label{locally-finite}
There is a unique algebra homomorphism, $\phi:\cO_q(GL_2)\hookrightarrow U$, defined on generators by $L\mapsto L^+S(L^-)$, i.e.
$$\phi: \ell^i_j\mapsto \sum_k l^{+i}_k S(l^{-k}_j).$$
Moreover, this is an algebra embedding, whose image is contained in the locally finite part of $U$ with respect to the adjoint action.
\end{proposition}
\begin{remark}
It was proved in \cite{JL} (see also \cite{KS}, 6.2.7) that the analogous map $\phi$ for $\mathfrak{sl}_2$ (more generally, for a simply connected semi-simple group) is an isomorphism onto the locally finite part.  A straightforward computation gives that $\phi(\det_q(L)) = (l^{+1}_1l^{+2}_2)^2$.  Hence, while the image of $\phi$ does not contain the central element,
$$K_1K_2=l^{+1}_1l^{+2}_2=\sqrt{\operatorname{det}_q(L)},$$ comparison with Rosso's isomorphism for $\mathfrak{sl}_2$ gives that the locally finite part is generated by the image of $\phi$ and this central element.  The entire quantum group is furthermore generated by adjoining $K_1$.
\end{remark}

Using this embedding, we can pull back the restriction of the adjoint action of $U$ on the locally finite part of $U$, and get an action of $U$ on $\cO_q(GL_2)$. Unpacking these definitions and identifications, we have the following:

\begin{lemma}\label{action-EFK-aij}
The embedding $\cO_q(GL_2)\hookrightarrow U$ of the quantum coordinate algebra into the quantum enveloping algebra and the resulting restriction of the adjoint action are given on generators by the following formulas. 
\begin{itemize}
\item The embedding: 
\begin{align*}
\ell^1_1&=K_1^{-2}+q^{-1}(q-q^{-1})^2K_1^{-1}K_2^{-1}EF, &  \ell^1_2&=q^{-1}(q-q^{-1})K_1^{-1}K_2^{-1}E,\\
\ell^2_1&=(q-q^{-1})K_2^{-2}F,  & \ell^2_2&=K_2^{-2}.
\end{align*}
\item The action: 
\begin{align*}
K_m\rhd \ell^i_j&=q^{\delta_{im}-\delta_{jm}}\ell^i_j&\\
E\rhd \ell^i_j&=\delta_{j=1}\ell^i_2-\delta_{i=2}q^{2\delta_{j=2}}\ell^1_j\\
F\rhd \ell^i_j&=\delta_{j=2}q^{2\delta_{i=2}-1}a^i_1-\delta_{i=1}q^{-1}\ell^2_j.
\end{align*}
\end{itemize}
\end{lemma}

\begin{lemma}\label{trdetarecenter}
The center of $\cO_q(GL_2)$ is equal to the space $\cO_q(GL_2)^U$ of $U$-invariants. It is the subalgebra generated by the q-trace ${\tr}_q(L)$, the q-determinant ${\det}_q(L)$, and its inverse, where the q-trace and the q-determinant are defined as:
\begin{align}
\begin{aligned}\label{deftrdet}
{\tr}_q(L)&=\ell^1_1+q^{-2}\ell^2_2 \\
{\det}_q(L)&=\ell^1_1\ell^2_2-q^2\ell^1_2\ell^2_1.
\end{aligned}
\end{align}
\end{lemma}

\begin{proof} Because $\cO_q(GL_2)$ is a flat deformation of $\cO(G)$, it follows that the subalgebra of invariants is generated by a unique element in degree one, and two.  The formulas here follow by direct computation.
\end{proof}

\begin{remark} Up to changes in convention, these formulas first appeared in \cite{K}.  For $N>2$, the description of the center of $\cO_q(GL_N)$, and hence $U_q\mathfrak{gl}_N$, in terms of q-deformed minors is taken up in \cite{JW}.
\end{remark}

\subsection{The algebra $\cD_q(GL_2)$ of quantum differential operators}

Next, we recall the definition of the algebra of polynomial quantum differential operators on $GL_2$. All the conventions follow \cite{J}, and correspond to the quiver $\overset{v}{\bullet}\overset{e}{\rightloop}$ in their notation.  They first appeared in this formulation in \cite{VV}.

\begin{definition}\label{def-D_q}
We let $\cD^+_q(GL_2)$ denote the algebra with:
\begin{itemize} 
\item generators $a^i_j$ and $\partial^i_j$, $i,j=1,2$, organized in matrices $$A=\left[\begin{array}{cc}a^1_1 & a^1_2 \\ a^2_1 & a^2_2 \end{array}\right]\quad \textrm{and}\quad D=\left[\begin{array}{cc}\partial^1_1 & \partial^1_2 \\ \partial^2_1 & \partial^2_2 \end{array}\right];$$
\item relations given by the entries of the following matrix equations
\begin{align}
R_{21}A_1RA_2&=A_2R_{21}A_1R \label{A-cross-matrix} \\
R_{21}D_1RD_2&=D_2R_{21}D_1R\label{D-cross-matrix} \\
R_{21}D_1RA_2&=A_2R_{21}D_1R_{21}^{-1}, \label{AD-cross-matrix}
\end{align}

\item which may be written out explicitly as follows:

\begin{align}
\begin{aligned}\label{relns-aa}
&a^1_2a^1_1 = a^1_1a^1_2 + (1-q^{-2})a^1_2a^2_2 & \quad  &a^2_2a^1_1 = a^1_1a^2_2& &\\
&a^2_1a^1_1 = a^1_1a^2_1 - (1-q^{-2})a^2_2a^2_1& \quad  &a^2_2a^1_2 = q^2a^1_2a^2_2& &\\
&a^2_1a^1_2 = a^1_2a^2_1 + (1-q^{-2})(a^1_1a^2_2 - a^2_2 a^2_2)& \quad  &a^2_2a^2_1 = q^{-2}a^2_1a^2_2\\
\end{aligned}
\end{align}

\begin{align}
\begin{aligned}\label{relns-dd}
&\d^1_2\d^1_1 = \d^1_1\d^1_2 + (1-q^{-2})\d^1_2\d^2_2 & \quad  &\d^2_2\d^1_1 = \d^1_1\d^2_2& &\\
&\d^2_1\d^1_1 = \d^1_1\d^2_1 - (1-q^{-2})\d^2_2\d^2_1& \quad  &\d^2_2\d^1_2 = q^2\d^1_2\d^2_2& &\\
&\d^2_1\d^1_2 = \d^1_2\d^2_1 + (1-q^{-2})(\d^1_1\d^2_2 - \d^2_2 \d^2_2)& \quad  &\d^2_2\d^2_1 = q^{-2}\d^2_1\d^2_2\\
\end{aligned}
\end{align}

\begin{align}
\begin{aligned}\label{relns-da}
&\d^1_1\cdot a^1_1= - (1-q^{-2})\cdot \d^1_2\cdot a^2_1 + q^{-2}\cdot a^1_1\cdot \d^1_1 + (q^{-2}-q^{-4})\cdot a^1_2\cdot \d^2_1\\
&\d^1_1\cdot a^1_2= (q^{-2}-1)\cdot \d^1_2\cdot a^2_2 + q^{-2}\cdot a^1_2\cdot \d^1_1\\
&\d^1_1\cdot a^2_1= (1-q^2)\cdot \d^2_1\cdot a^1_1 - (q-q^{-1})^2\cdot \d^2_2\cdot a^2_1 + a^2_1\cdot \d^1_1 + (1-q^{-2})\cdot a^2_2\cdot \d^2_1\\
&\d^1_1\cdot a^2_2=  (1-q^2)\cdot \d^2_1\cdot a^1_2 - (q-q^{-1})^2\cdot \d^2_2\cdot a^2_2 + a^2_2\cdot \d^1_1\\
&\d^1_2\cdot a^1_1= a^1_1\cdot \d^1_2 + (1-q^{-2})\cdot a^1_2\cdot \d^2_2 + (q^{-2} - 1)\cdot a^1_2\cdot \d^1_1\\
&\d^1_2\cdot a^1_2=  q^{-2}\cdot a^1_2\cdot \d^1_2\\
&\d^1_2\cdot a^2_1=  (q^{-2}-1)\cdot \d^2_2\cdot a^1_1 + a^2_1\cdot \d^1_2 + (q^{-2}-1)\cdot a^2_2\cdot \d^1_1 + (1-q^{-2})\cdot a^2_2\cdot \d^2_2\\
&\d^1_2\cdot a^2_2=  (q^{-2}-1)\cdot \d^2_2\cdot a^1_2 + q^{-2}\cdot a^2_2\cdot \d^1_2\\
&\d^2_1\cdot a^1_1= q^{-2}\cdot a^1_1\cdot \d^2_1 - (1-q^{-2})\cdot \d^2_2\cdot a^2_1\\
&\d^2_1\cdot a^1_2=  (q^{-2}-1)\cdot \d^2_2\cdot a^2_2 + a^1_2\cdot \d^2_1\\
&\d^2_1\cdot a^2_1=  q^{-2}\cdot a^2_1\cdot \d^2_1\\
&\d^2_1\cdot a^2_2=  a^2_2\cdot \d^2_1\\
&\d^2_2\cdot a^1_1=  a^1_1\cdot \d^2_2 + (1-q^2)\cdot a^1_2\cdot \d^2_1\\
&\d^2_2\cdot a^1_2=  a^1_2\cdot \d^2_2\\
&\d^2_2\cdot a^2_1=  q^{-2}\cdot a^2_1\cdot \d^2_2 + (q^{-2}-1)\cdot a^2_2\cdot \d^2_1\\
&\d^2_2\cdot a^2_2=  q^{-2}\cdot a^2_2\cdot \d^2_2.
\end{aligned}
\end{align}
\end{itemize}
\end{definition}

\begin{proposition} \label{Dq-PBW-prop}
\begin{enumerate}
\item  The elements
$${\det}_q(A):=a^1_1a^2_2-q^2a^1_2a^2_1, \qquad  {\det}_q(D):=\d^1_1\d^2_2-q^2\d^1_2\d^2_1$$
satisfy:
\begin{align*}
{\det}_q(A)a^i_j=a^i_j{\det}_q(A) \qquad & {\det}_q(D)a^i_j=q^{-2}a^i_j{\det}_q(D)\\
 \partial^i_j {\det}_q(A)=q^{-2}{\det}_q(A)\partial^i_j\qquad &  \partial^i_j {\det}_q(D)={\det}_q(D)\partial^i_j.
\end{align*}
In particular, the multiplicative set $\{{\det}_q(A)^m{\det}_q(D)^n| m,n\in \mathbb{N} \}$ satisfies the Ore condition. 

\item The set of elements 
$$a^{i_1}_{j_1}a^{i_2}_{j_2}\ldots a^{i_m}_{j_m}\d^{k_1}_{l_1}\d^{k_2}_{l_2}\ldots a^{k_n}_{l_n}$$ 
with $n,m\in \mathbb{N}_0$, $i_p,j_p,k_p,l_p \in \{1,2\}$, subject to the conditions
$$\forall p \quad (i_p<i_{p+1}) \textrm{ or } (i_p=i_{p+1} \textrm{ and } j_p\le j_{p+1}) $$
$$\forall p \quad (k_p<k_{p+1}) \textrm{ or } (k_p=k_{p+1} \textrm{ and } l_p\le l_{p+1}) $$
forms a basis of  $\cD^+_q(GL_2)$.
\end{enumerate}
\end{proposition}
\begin{proof}
Claim (1) is a straightforward computation; claim (2) is Theorem 5.3 in \cite{J}.

\end{proof}

\begin{definition}We let $\cD_q(GL_2)$ denote the algebra,
\begin{align*} 
\cD_q(GL_2)&:=\cD^+_q(GL_2)[{\det}_q(A)^{-1}{\det}_q(D)^{-1}],
\end{align*}
obtained by localizing $\cD^+_q(GL_2)$ at the element ${\det}_q(A){\det}_q(D)$. We call this the algebra of polynomial quantum differential operators on $GL_2$.
\end{definition}

\begin{remark}The two evident copies of $\cO_q(GL_2)$ inside $\cD_q(GL_2)$, embedded by $L\mapsto A$ and $L\mapsto D$, carry an action of $U$ as described in Lemma \ref{action-EFK-aij}. This extends to the action of $U$ on the algebra $\cD_q(GL_2)$, a consequence of the construction of $\cD_q(GL_2)$ via a $U$-linear map, see Section 3 in \cite{J}. 
\end{remark}

The algebra $\cD_q(GL_2)$ is $\mathbb{Z}^2$-graded by 
\begin{align}
\label{grading-Dq}
{\deg}(a^i_j)=(1,0), \quad  {\deg}(\partial^i_j)=(0,1). 
\end{align}
This grading is inner, in the sense that for any $h\in \cD_q(GL_2)$ with $\deg(h)=(M,N)$ we have
\begin{align}
\label{grading-Dq-inner}
{\det}_q(A) \cdot h= q^{2N}h \cdot {\det}_q(A), \qquad {\det}_q(D) \cdot  h= q^{-2M}h \cdot {\det}_q(D).
\end{align}

\subsection{The quantum Hamiltonian reduction $\cD_q(GL_2)\dS_{\!\!\! \mathcal{I}_t} U$} In addition to the two embeddings of $\cO_q(GL_2)$ into  $\cD_q(GL_2)$ given by $L\mapsto A$ and $L\mapsto D$, we will use a third, more involved embedding, given by the ``quantum moment map". We will make use of the following quantum cofactor matrices $\widetilde{D}$ and $\widetilde{A}$, and the quantum inverse matrices, $D^{-1}$ and $A^{-1}$:
\begin{align}
&\widetilde{A}=\left[\begin{array}{cc}a^2_2 & -q^2a^1_2 \\ -q^2a^2_1 & q^2a^2_2+(1-q^2)a^2_2 \end{array}\right]  && \widetilde{D}=\left[\begin{array}{cc}\partial^2_2 & -q^2\partial^1_2 \\ -q^2\partial^2_1 & q^2\partial^2_2+(1-q^2)\partial^2_2 \end{array}\right] \notag\\
\label{def-tilde}
&A^{-1}=({\det}_q(A))^{-1}\tilde{A}  &&  D^{-1}=({\det}_q(D))^{-1}\tilde{D}.
\end{align}
The notation $A^{-1}$ and $D^{-1}$ is justified by the fact that
$AA^{-1}=A^{-1}A=I$ and $DD^{-1}=D^{-1}D=I.$
\begin{def/prop}[\cite{J}, Proposition 7.20, Definition 7.24]\label{def-muq}
We have a homomorphism of algebras,
$$\mu_q:\cO_q(GL_2) \to \cD_q(GL_2)\quad \textrm{defined on generators by} \quad \mu_q(L)=DA^{-1}D^{-1}A.$$
The homomorphism $\mu_q$ is called the \emph{quantum moment map}. 
\end{def/prop}

For $t\in \CC^\times$, let $X_t\in GL_2$ denote the matrix,
$$X_t:=\left(\begin{array}{cc}t^{-2}&0\\0&t^{2}\end{array}\right),$$
and let $O_t$ denote its conjugacy class in $GL_2$. 
The multiplicative Calogero-Moser variety is the set
$$C_t := \Big\{ (A,B) \in GL_2\times GL_2, \,\, \textrm{ such that } AB^{-1}A^{-1}B\in O_t\Big\}.$$

The $GL_2$-varieties $O_t$ and $C_t$ admit canonical equivariant $q$-deformations, constructed as follows.

\begin{definition}\label{def-Zt}
For $t\in\CC^\times$, let $\mathcal{I}_t\subset \cO_q(G)$ denote the two-sided ideal in $\cO_q(G)$ generated by the central elements:
$$Z_t:=\tr_q(L-q^4\cdot X_t) = \tr_q(L)-q^4\cdot(t^{-2}+q^{-2}t^2),\qquad \textrm{ and } \textrm{det}_q(L)-q^8.$$ 
Consider also the left $\cD_q(G)$ ideal  $\cD_q(G)\cdot \mu_q(\mathcal{I}_t)$. 
\end{definition}

\begin{remark}
The two-sided ideal $\mathcal{I}_t$ and the left ideal $\cD_q(G)\cdot \mu_q(\mathcal{I}_t)$ are ad-equivariant $q$-deformations of the defining ideals of $O_t$ and $C_t$, respectively. 
\end{remark}

We note that $\mu_q({\det}_q(L)-q^8)=0$, so that 
$$\cD_q(G)\cdot \mu_q(\mathcal{I}_t)=\cD_q(G)\cdot \mu_q(Z_t)$$ 
is in fact a principal left ideal. Since the principal generator $$\mu_q(Z_t)=\mathrm{tr}_q(\mu_q(L))-q^4\cdot(t^{-2}+q^{-2}t^2)$$ 
is $U$-invariant by Lemma \ref{trdetarecenter}, it follows that $\cD_q(G)\cdot \mu_q(Z_t)$ is preserved by $U$, and the action of $U$ descends to the quotient,
$$M:=\cD_q(GL_2)\Big/\cD_q(GL_2)\cdot \mu_q(Z_t).$$
\begin{definition}
The quantum Hamiltonian reduction of $\cD_q(GL_2)$ by the moment map $\mu_q$ along the orbit $O_t$ is defined to be the subspace
$$\cD_q(GL_2)\underset{\mathcal{I}_t}{\dS}U:=\left(\cD_q(GL_2)\Big/\cD_q(GL_2)\cdot \mu_q(Z_t)\right)^U$$ 
of $U$-invariants in the quotient of $\cD_q(GL_2)$ by the left ideal generated by $\mu_q(\mathcal{I}_t)$.
\end{definition}

Any endomorphism $\rho$ of the cyclic $\cD_q(GL_2)$-module $M$ is uniquely determined by $\rho(1)$, which should be invariant for the $U$-action, and should have the property that $\mathcal{I}_t\cdot \rho(1) \subset \mathcal{I}_t$.  The quantum moment map condition, $\mu(x).y = (x_{(1)} \cdot y) \mu(x_{(2)})$ ensures that the former property implies the latter, as it implies that $\mathcal{I}_t$ commutes with $\rho(1)$. Hence we have an isomorphism,
$$\cD_q(GL_2)\underset{\mathcal{I}_t}{\dS}U \cong \operatorname{End}\left(\cD_q(GL_2)\Big/\cD_q(GL_2)\cdot\mu_q(Z_t)\right).$$
We therefore regard the quantum Hamiltonian reduction as an algebra, with this multiplication.   This algebra structure can be understood more directly, as follows:

\begin{lemma}\label{ReduceToInvariants}
We have a natural isomorphism of algebras:
$$\mathcal{D}_q(GL_2)\dS_{\!\!\!\mathcal{I}_t} U\cong \mathcal{D}_q(GL_2)^U\Big/ \mathcal{D}_q(GL_2)^U \cdot \mu_q(Z_t).$$
\end{lemma}
\begin{proof}
Locally finite representations of $U$ are semisimple, so taking invariants is exact and we have
$$\mathcal{D}_q(GL_2)\dS_{\!\!\!\mathcal{I}_t} U=\left(\cD_q(GL_2)\Big/\cD_q(GL_2)\cdot \mu_q(Z_t)\right)^U \cong \cD_q(GL_2)^U\Big/\left(\cD_q(GL_2)\cdot \mu_q(Z_t)\right)^U.$$

As the generator $\mu_q(Z_t)$ of the ideal is $U$-invariant, we also have 
$$\left(\cD_q(GL_2)\cdot \mu_q(Z_t)\right)^U=\cD_q(GL_2)^U \cdot \mu_q(Z_t).$$
\end{proof}

The usefulness of this lemma is that it reduces the problem of finding the presentation of $\mathcal{D}_q(GL_2)\dS_{\!\!\!\mathcal{I}_t} U$ to the problem of finding the presentation of $\mathcal{D}_q(GL_2)^U$. This is employed in Section \ref{PBW-DqG}.

The ideal $\cD_q(GL_2)\cdot\mu_q(Z_t)$ is only a one-sided ideal in $\cD_q(GL_2)$.  However, its sub-space of invariants, $\left(\cD_q(GL_2)\cdot\mu_q(Z_t)\right)^U$, is a two-sided ideal in $\cD_q(GL_2)^U$. This equips $\mathcal{D}_q(GL_2)\dS_{\!\!\!\mathcal{I}_t} U$ with an algebra structure, which coincides with the above description of it as the algebra of endomorphisms of the module $M$. The following theorem gives a presentation of this algebra by generators and relations, and the corresponding PBW basis for it. 



\begin{theorem} \label{HamiltonPresentation}
Assume that $q,t\neq 0$, and that $q$ is not a nontrivial root of unity. 

\begin{enumerate}
\item The quantum Hamiltonian reduction $\cD_q(GL_2)\dS_{\!\!\! \mathcal{I}_t} U$ is isomorphic to the algebra with generators $c_1,c_2^{\pm1},d_1,d_2^{\pm1},r$ and relations:\\
\begin{minipage}[t]{0.4\textwidth}
\begin{align}
c_2c_1&=c_1c_2 \\
d_2d_1&=d_1d_2  \\
d_2c_2&=q^{-4}c_2d_2 \\
d_2c_1&=q^{-2}c_1d_2\\
d_1c_2&=q^{-2}c_2d_1 
\end{align}
\end{minipage}
\begin{minipage}[t]{0.6\textwidth}
 \begin{align}
d_2r&= q^{-2}rd_2 \\
rc_2&= q^{-2}c_2r \\
d_1r&=q^{-2}rd_1+(1-q^{-2})q^{-2}c_1d_2  \\
rc_1&= q^{-2}c_1r+(1-q^{-2})q^{-2}c_2d_1\\
d_1c_1&=c_1d_1+(q^{-2}-1)r 
\end{align}
\end{minipage}

\begin{equation} r^2=q^{-4}(1+t^2)(q^{-2}+t^{-2})c_2d_2-
q^{-4}c_2d_1^2- q^{-4}c_1^2d_2+q^{-2}c_1rd_1.\end{equation}

The isomorphism is given by 
\begin{align}
\begin{aligned}
\Psi(c_1)= \tr_q(A), \quad &  \Psi(c_2)=  {\det}_q(A), \quad & \Psi(r) = q^2  \tr_q(DA), \\
\Psi(d_1)= \tr_q(D), \quad &  \Psi(d_2)= {\det}_q(D). 
\end{aligned}
\label{Dq-pres-iso}
\end{align}

\item After the identification given by $\Psi$, the set  $$\{c_1^{a_1}c_2^{a_2}r^\epsilon d_1^{b_1}d_2^{b_2} \ | \  \epsilon \in \{0,1\}, a_1,b_1 \in \mathbb{N}_0, a_2,b_2 \in \mathbb{Z} \}$$
forms a basis of the algebra $\cD_q(GL_2)\dS_{\!\!\! \mathcal{I}_t} U$.
\end{enumerate}

\end{theorem}

Section \ref{PBW-DqG} is dedicated to the proof of this Theorem.

\section{The quantum Harish-Chandra isomorphism}\label{mainstatement}

\begin{theorem}\label{main-result}

Assume that $q,t\neq 0$, $t^2+1\ne 0$, and $q$ is not a nontrivial root of $1$. There exists a unique graded isomorphism of algebras,
$$HC_{t,q}:\cD_q(GL_2)\dS_{\!\!\! \mathcal{I}_t} U\xrightarrow{\cong} e\mathbb{H}_{t,q}e$$
such that:
\begin{align*}
HC_{t,q}(d_1)=P_1,  \quad \quad \qquad & HC_{t,q}(d_2)=q^2P_2, &  HC_{t,q}(r)=R, \\
HC_{t,q}(c_1)=Q_1,  \quad \quad \qquad &  HC_{t,q}(c_2)=q^2Q_2.\\
\end{align*}
\end{theorem}
\begin{proof}
This follows directly by comparing the presentation of $ e\mathbb{H}_{t,q}e$ given in Theorem \ref{thm-sDAHApresentation} and the presentation of $\cD_q(GL_2)\dS_{\!\!\! \mathcal{I}_t} U$ given in Theorem \ref{HamiltonPresentation}. 
\end{proof}

\begin{remark}  We note that the isomorphism $HC_{t,q}$ extends, even for $q$ a root of unity, to give an isomorphism:
$$HC_{t,q}: \mathcal{B}\Big/\mathcal{B}\cdot(w-(t^{-2} + q^{-2}t^2)c_2d_2) \xrightarrow{\cong} e\HH e,$$
where $\mathcal{B}$ is an explicitly given algebra with generators and relations (see Prop \ref{PresentationInvariants}), which recovers the Hamiltonian reduction when $q$ is not a root of unity.  In the case the $q$ is a non-trivial root of unity, one should take more care with defining the Hamiltonian reduction, so that the isomorphism will still hold.  It should be interesting to compare this directly to \cite{VV}.
\end{remark}

\section{Proof of Theorem \ref{thm-sDAHApresentation}} \label{sec-DAHAPResentation}

This section contains the proof of Theorem \ref{thm-sDAHApresentation}. The aim is to give a presentation of the spherical DAHA by generators and relations (Theorem \ref{thm-sDAHApresentation} (1)), and at the same time to find a PBW type basis for it (Theorem \ref{thm-sDAHApresentation} (2)). 

Throughout this section, let $\mathcal{A}$ be the algebra 
with generators $P_1$, $P_2^{\pm1}$, $Q_1$, $Q_2^{\pm1}$, $R$ and relations \eqref{P1P2}-\eqref{R2reln}. 
It is $\mathbb{Z}^2$-graded by 
$$\deg(P_1)=(0,1), \,\,\, \deg(P_2)=(0,2), \,\,\, \deg(Q_1)=(1,0),\,\,\, \deg(Q_2)=(2,0), \,\,\, \deg(R)=(1,1).$$ We first describe a PBW basis of this algebra. 

\begin{lemma}\label{PBW-A}
The set $$\{Q_1^{a_1}Q_2^{a_2}R^\epsilon P_1^{b_1}P_2^{b_2} \ | \  \epsilon \in \{0,1\}, a_1,b_1 \in \mathbb{N}_0, a_2,b_2 \in \mathbb{Z} \}$$
forms a basis of the algebra $\mathcal{A}$.
\end{lemma}
\begin{proof}
We will use the Diamond Lemma from \cite{B}. In the language of that paper, the defining relations \eqref{P1P2}-\eqref{R2reln} of $\mathcal{A}$ are \emph{straightening} relations, prescribing how to replace (straighten) certain monomials by linear combinations of other, simpler (with respect to some ordering), monomials.

For a monomials $ABC$ such that $AB$ can be straightened and $BC$ can be straightened, we say that \emph{the straightening diamond} holds if the two resulting straightenings of $ABC$ can be further straightened to a common value:  in other words, the results of straightening $ABC$ should be independent on the choice to first straighten $AB$ or to first straighten $BC$. As proved in \cite{B}, the necessary and sufficient condition for such a straightening algorithm to be well defined, and for the resulting set of straightened monomials to be a basis, is that the straightening diamonds hold for all monomials $ABC$.

To apply this theorem to the algebra $\mathcal{A}$, we must establish the straightening diamonds for all of the following monomials:
$$P_2P_1R, \quad P_2P_1Q_2,\quad P_2P_1Q_1,\quad P_2RQ_2,\quad P_2RQ_1,$$ 
$$P_2Q_2Q_1,\quad  P_1RQ_2, \quad R^2Q_2,\quad P_1Q_2Q_1,\quad RQ_2Q_1,$$ 
$$P_2R^2, \quad   P_1RQ_1, \quad P_1R^2, \quad R^2Q_1, \quad R^3.$$

This is done by direct computation. Because $P_2$ and $Q_2$ each $q$-commute with all other generators, the only non-trivial checks involve the final four monomials. For illustration, we prove that the straightening diamond for $P_1RQ_1$ holds. Straightening $P_1R$ first, we get
\begin{align*}
(P_1R)Q_1&\stackrel{\eqref{P1R}}{=} q^{-2}RP_1Q_1+(1-q^{-2})Q_1P_2Q_1\\
&\hspace{-0.32cm}\stackrel{\eqref{P1Q1},\eqref{P2Q1}}{=} q^{-2}RQ_1P_1+q^{-2}(q^{-2}-1)R^2+(1-q^{-2})q^{-2}Q_1^2P_2\\
&\stackrel{\eqref{RQ1}}{=} q^{-4}Q_1RP_1+q^{-2}(1-q^{-2})Q_2P_1^2+q^{-2}(q^{-2}-1)R^2+q^{-2}(1-q^{-2})Q_1^2P_2.
\end{align*}
Straightening $RQ_1$ first instead, we get
\begin{align*}
P_1(RQ_1)&\stackrel{\eqref{RQ1}}{=} q^{-2}P_1Q_1R+(1-q^{-2})P_1Q_2P_1\\
&\hspace{-0.32cm}\stackrel{\eqref{P1Q1},\eqref{P1Q2}}{=} q^{-2}Q_1P_1R+q^{-2}(q^{-2}-1)R^2+q^{-2}(1-q^{-2})Q_2P_1^2\\
&\stackrel{\eqref{P1R}}{=}q^{-4}Q_1RP_1+q^{-2}(1-q^{-2})Q_1^2P_2+q^{-2}(q^{-2}-1)R^2+q^{-2}(1-q^{-2})Q_2P_1^2.
\end{align*}
By inspection, these two expressions agree. 

This proves the corresponding claim for the subalgebra $\mathcal{A}^+$ generated by $P_1$, $P_2$, $Q_1$, $Q_2$, $R$. The claim for $\mathcal{A}$ follows by localization at $P_2Q_2$.

\end{proof}

Next, we show that the defining relations of $\mathcal{A}$ are satisfied in the spherical DAHA.

\begin{lemma}\label{Phi-is-hom}
The map $\Phi:\mathcal{A}\to e\mathbb{H}_{t,q}e$ given by \eqref{sDAHA-pres-iso} is a graded algebra homomorphism.
\end{lemma}
\begin{proof}
The images of $P_1,P_2,R,Q_1,Q_2$ under the proposed map $\Phi$ lie in the spherical subalgebra of $\mathbb{H}_{t,q}$, and the images of $P_2$ and $Q_2$ are invertible. It remains to show that $\Phi(P_1),\Phi(P_2),\Phi(Q_1),\Phi(Q_2)$ and $\Phi(R)$ satisfy the relations \eqref{P1P2}-\eqref{R2reln}. This is a straightforward, if lengthy computation, using relations from Definition \ref{defDAHA} and Proposition \ref{prop-AllRelnsDAHA}.

First, notice that in  $\Phi(P_1),\Phi(P_2),\Phi(Q_1),\Phi(Q_2), \Phi(R)$, it is enough to write $e$ on the left:
\begin{align}\label{eP1e=eP1=P1e}
\begin{aligned}
&\Phi(P_1)=e(X_1+X_2)e=(X_1+X_2)e=e(X_1+X_2)  \\
&\Phi(Q_1)=e(Y_1+Y_2)e=(Y_1+Y_2)e=e(Y_1+Y_2) \\
&\Phi(P_2)=e(X_1X_2)e=(X_1X_2)e=e(X_1X_2)\\
&\Phi(Q_2)=e(Y_1Y_2)e=(Y_1Y_2)e=e(Y_1Y_2)\\
&\Phi(R)=e(t^{-2}Y_1X_1+Y_2X_2)e= e(t^{-2}Y_1X_1+Y_2X_2). 
\end{aligned}
\end{align}
This simplifies some computations by allowing us to calculate in the subalgebra generated by $X_1,X_2,Y_1,Y_2$. 

For example, using the relation $X_1X_2=X_2X_1$ from Definition \ref{defDAHA}, we have that 
\begin{align*}
\Phi(P_2P_1)&=e(X_1X_2)e\cdot e(X_1+X_2)e\\
&=e(X_1X_2)(X_1+X_2)e\\
&=e(X_1+X_2)(X_1X_2)e\\
&=\Phi(P_1P_2), 
\end{align*}
proving \eqref{P1P2}. In a similar way, the relation \eqref{Q1Q2} follows from $Y_1Y_2=Y_2Y_1$, and relations \eqref{P2Q2}, \eqref{P2Q1}, \eqref{P1Q2}, \eqref{P2R} and \eqref{RQ2} follow from \eqref{X1X2-is-q-central}. For illustration we include the slightly more involved proof of relation \eqref{P1Q1}; relations \eqref{P1R}, \eqref{RQ1} and \eqref{R2reln} are proved analogously. 
\begin{align*}
\Phi(P_1Q_1)&\stackrel{\phantom{\eqref{eT=Te=te}}}{=} e(X_1+X_2)e\cdot e (Y_1+Y_2)e\\
&\stackrel{\eqref{eP1e=eP1=P1e}}{=}e(X_1Y_1+X_1Y_2+X_2Y_1+X_2Y_2)e\\
&\stackrel{\eqref{defDAHA-relations2}}{=} e\left( q^{-2}T^{-2}Y_1X_1+Y_2X_1 +(t-t^{-1}) T^{-1}Y_1X_1+\right.\\
&\qquad \left. +Y_1X_2 +(t-t^{-1})q^{-2}T^{-1} Y_1X_1+q^{-2}Y_2X_2T^{-2}\right)e\\
&\stackrel{\eqref{defDAHA-relations3}}{=} e\left( q^{-2}T^{-2}Y_1X_1+Y_2X_1 +(t-t^{-1})T^{-1}Y_1X_1+\right.\\
&\qquad \left. +Y_1X_2 +(t-t^{-1})q^{-2}T^{-1} Y_1X_1+q^{-2}Y_2X_2-(t-t^{-1})q^{-2}T^{-1}Y_1X_1\right)e\\
&\stackrel{\eqref{eT=Te=te}}{=} e\left((1-t^{-2}+q^{-2}t^{-2})Y_1X_1+Y_1X_2+Y_2X_1+q^{-2}Y_2X_2\right)e\\
&\stackrel{\phantom{\eqref{eT=Te=te}}}{=} e\left((Y_1+Y_2)(X_1+X_2)+(q^{-2}-1)(t^{-2}Y_1X_1+Y_2X_2)\right)e\\
&\stackrel{\phantom{\eqref{eT=Te=te}}}{=}\Phi(Q_1P_1+(q^{-2}-1)R).
\end{align*}
\end{proof}

This proves that $\Phi(P_1), \Phi(P_2),\Phi(Q_1),\Phi(Q_2), \Phi(R)$ really satisfy the relations \eqref{P1P2}-\eqref{R2reln} (i.e.~that $\Phi$ is a homomorphism). It remains to show that they generate the whole spherical double affine Hecke algebra (i.e.~that $\Phi$ is surjective), and that the stated relations are exhaustive (i.e.~that $\Phi$ is injective). We tackle injectivity first. 

\begin{lemma}\label{Phi-inj}
The set 
$$\{\Phi(Q_1^{a_1}Q_2^{a_2}R^\epsilon P_1^{b_1}P_2^{b_2}) \ | \  \epsilon \in \{0,1\}, a_1,b_1 \in \mathbb{N}_0, a_2,b_2 \in \mathbb{Z} \}$$
is linearly independent in $e\mathbb{H}_{t,q}e$. The homomorphism $\Phi$ is injective.
\end{lemma}
\begin{proof}
We will use the PBW theorem for $\mathbb{H}_{t,q}$ (Proposition \ref{prop-AllRelnsDAHA} (2)) to prove that the above set is linearly independent in $e\mathbb{H}_{t,q}e\subseteq \mathbb{H}_{t,q}$. 

Assume that for some finite indexing set $I$, some collection of nonzero scalars $\alpha_i$, $i\in I$, and some $\epsilon_i \in \{0,1\}, a_{i,1},b_{i,1} \in \mathbb{N}_0, a_{i,2},b_{i,2} \in \mathbb{Z}$, $i\in I$, we have 
$$\sum_{i\in I} \alpha_i \Phi(Q_1^{a_{i,1}}Q_2^{a_{i,2}}R^{\epsilon_i} P_1^{b_{i,1}}P_2^{b_{i,2}})=0.$$
Using the definition of the map $\Phi$ and the observation \eqref{eP1e=eP1=P1e}, we can write this as
$$\sum_{i\in I} \alpha_i \, e \, (Y_1+Y_2)^{a_{i,1}}(Y_1Y_2)^{a_{i,2}}(t^{-2}Y_1X_1+Y_2X_2)^{\epsilon_i} (X_1+X_2)^{b_{i,1}}(X_1X_2)^{b_{i,2}}=0.$$
Using $e=\frac{1+tT}{1+t^2}$ and the PBW theorem for DAHA, it follows that also
\begin{align}
\sum_{i\in I} \alpha_i (Y_1+Y_2)^{a_{i,1}}(Y_1Y_2)^{a_{i,2}}(t^{-2}Y_1X_1+Y_2X_2)^{\epsilon_i} (X_1+X_2)^{b_{i,1}}(X_1X_2)^{b_{i,2}}=0.\label{lin-alg-1}
\end{align}
The subalgebra generated by $Y_1^{\pm1},Y_2^{\pm1}$ is commutative. After multiplying \eqref{lin-alg-1} on the left by some power of $Y_1Y_2$ if necessary, we can assume that $a_{i,2}\ge 0$ for all $i\in I$, and that there exists at least one $i\in I$ for which $a_{i,2}=0$. Similarly, by multiplying \eqref{lin-alg-1} on the right by some power of $X_1X_2$, we can assume that $b_{i,2}\ge 0$ for all $i\in I$.

As explained after Definition \ref{defDAHA}, the algebra $\mathbb{H}_{t,q}$ is $\mathbb{Z}^2$ graded. We may assume that all terms in the expression \eqref{lin-alg-1} are of the same bigraded degree $(M,N)$. This means that for all $i\in I$ we have 
\begin{align*}
a_{i,1}+2a_{i,2}+\epsilon_i&=M\\
b_{i,1}+2b_{i,2}+\epsilon_i&=N.
\end{align*}

The PBW theorem for DAHA (Proposition \ref{prop-AllRelnsDAHA} (2)) implies that, as a vector space,  $\mathbb{H}_{t,q}$ is also $\mathbb{Z}^4$ graded by $(\deg(Y_1),\deg(Y_2), \deg(X_1), \deg(X_2))$. The highest power of $Y_1$ in the term $\alpha_i (Y_1+Y_2)^{a_{i,1}}(Y_1Y_2)^{a_{i,2}}(t^{-2}Y_1X_1+Y_2X_2)^{\epsilon_i} (X_1+X_2)^{b_{i,1}}(X_1X_2)^{b_{i,2}}$ of the expression \eqref{lin-alg-1} is equal to 
$$a_{i,1}+a_{i,2}+\epsilon_i=M-a_{i,2}.$$
Using the assumption that there exists an $i\in I$ with $a_{i,2}=0$ and the PBW theorem for $\mathbb{H}_{t,q}$, we get that the coefficient of $Y_1^{M}$ in \eqref{lin-alg-1} is  
\begin{align}
\sum_{\substack{i\in I \\ a_{i,2}=0\\ \epsilon_i=0}} \alpha_i (X_1+X_2)^{b_{i,1}}(X_1X_2)^{b_{i,2}}+\sum_{\substack{i\in I \\ a_{i,2}=0\\ \epsilon_i=1}} \alpha_i t^{-2} X_1(X_1+X_2)^{b_{i,1}}(X_1X_2)^{b_{i,2}}=0,\label{lin-alg-3}
\end{align}
while at least one of the coefficients $\alpha_i$ in \eqref{lin-alg-3} is nonzero. Let 
\begin{align}
f=\sum_{\substack{i\in I \\ a_{i,2}=0\\ \epsilon_i=0}} \alpha_i (X_1+X_2)^{b_{i,1}}(X_1X_2)^{b_{i,2}}, \quad g=\sum_{\substack{i\in I \\ a_{i,2}=0\\ \epsilon_i=1}} \alpha_i t^{-2} (X_1+X_2)^{b_{i,1}}(X_1X_2)^{b_{i,2}}. \label{lin-alg-4}
\end{align}
Formula \eqref{lin-alg-3} states that
\begin{align} f+X_1g=0. \label{lin-alg-5}
\end{align}
Polynomials $f$ and $g$ are symmetric $X_1,X_2$, so symmetrizing \eqref{lin-alg-5} we get that 
$$f+\frac{X_1+X_2}{2}g=0.$$
Subtracting this from \eqref{lin-alg-5} we get that
$$(X_1-X_2)g=0.$$
This and the PBW basis imply that $g=0$, and so also $f=0$. 

At least one of the coefficients $\alpha_i$ in at least one of the expressions in \eqref{lin-alg-4} is nonzero. On the other hand, another use of the PBW theorem for DAHA implies that the set $\{ (X_1+X_2)^{b_{i,1}}(X_1X_2)^{b_{i,2}}\}$ is linearly independent. This is a contradiction.

\end{proof}

To show that the injective graded homomorphism $\Phi:\mathcal{A}\to \mathbb{H}_{t,q}$ is surjective, we would like to employ a dimension argument. However, these algebras are infinite dimensional, and bigraded with infinite dimensional bigraded pieces. To work around this problem, we introduce the following auxilary algebras. 

\begin{definition}\label{A+defn}
Let $\mathbb{H}_{t,q}^+$ denote the subalgebra of $\mathbb{H}_{t,q}$ generated by $T^{\pm1}, X_1,X_2,Y_1,Y_2$. Let $\mathcal{A}^+$ denote the subalgebra of $\mathcal{A}$ generated by $P_1,P_2,Q_1,Q_2$ and $R$.
\end{definition}

These subalgebras are also $\mathbb{Z}^2$ graded, and have finite dimensional bigraded pieces. The homomorphism $\Phi$ restricts to a graded homomorphism $\mathcal{A}^+\to e\mathbb{H}_{t,q}^+e$. By the next lemma, we can easily recover information about $\mathbb{H}_{t,q}$, $e\mathbb{H}_{t,q}e$ and $\mathcal{A}$ from the information about $\mathbb{H}_{t,q}^+$, $e\mathbb{H}_{t,q}^+e$ and $\mathcal{A}^+$.

\begin{lemma}\label{localiczationDAHA}
\begin{enumerate}
\item The multiplicative set $\{ (Y_1Y_2)^a(X_1X_2)^b\}$ in $\mathbb{H}_{t,q}$ satisfies the Ore condition, and the multiplicative set $\{ Q_2^aP_2^b\}$ in $\mathcal{A}$ satisfies the Ore condition.
\item The algebra $\mathbb{H}_{t,q}$ is the localization of $\mathbb{H}_{t,q}^+$ by $(Y_1Y_2)^{-1}(X_1X_2)^{-1}$, and the algebra $\mathcal{A}$ is the localization of $\mathcal{A}^+$ by $ Q_2^{-1}P_2^{-1}$.
\end{enumerate}
\end{lemma}

We will first prove that the restriction of the homomorphism $\Phi$  to $\mathcal{A}^+\to e\mathbb{H}_{t,q}^+e$ is an isomorphism, and then use this lemma to deduce the same about $\Phi:\mathcal{A}\to e\mathbb{H}_{t,q}e$. 

Let $\mathbf{A}$ be any algebra over $\mathbb{C}$ or over $R_{t,q}$ with a $\mathbb{Z}^2$ grading. If working over the ground ring $\mathbb{C}$, assume that every bigraded piece $\mathbf{A}[M,N]$ is a finite dimensional $\mathbb{C}$-vector space with dimension $\dim (\mathbf{A}[M,N])$. If working over the ground ring $R_{t,q}$, assume that every bigraded piece $\mathbf{A}[M,N]$ is a free $R_{t,q}$-modules of finite rank $\dim (\mathbf{A}[M,N])$. In either case, we record the dimensions of the bigraded pieces as a Hilbert series in two variables given by 
\begin{align*}
\mathrm{Hilb}_{\mathbf{A}}(u,v)=\sum_{(M,N)}\dim (\mathbf{A}[M,N])u^Mv^N.
\end{align*}
Let us compute the Hilbert series of the algebras $\mathcal{A}^+$ and $e\mathbb{H}_{t,q}^+e$.

\begin{lemma}\label{Hilb-A}
The set $$\{Q_1^{a_1}Q_2^{a_2}R^\epsilon P_1^{b_1}P_2^{b_2} \ | \  \epsilon \in \{0,1\}, a_i,b_i \in \mathbb{N}_0\}$$ is a basis of the algebra $\mathcal{A}^+$. The Hilbert series of this bigraded algebra is given by 
$$\mathrm{Hilb}_{\mathcal{A}^+}(u,v)=\frac{1}{1-u}\cdot \frac{1}{1-u^2}\cdot \frac{1}{1-v}\cdot \frac{1}{1-v^2}\cdot (1+uv).$$
\end{lemma}
\begin{proof}
The first assertion follows directly from Lemma \ref{PBW-A}. To calculate the Hilbert series, note that $\deg(Q_1^{a_1}Q_2^{a_2}R^\epsilon P_1^{b_1}P_2^{b_2})= (a_1+2a_2+\epsilon,b_1+2b_2+\epsilon)$, so 
\begin{align*}
\mathrm{Hilb}_{\mathcal{A}^+}(u,v)&=\sum_{\substack{a_1,a_2\ge 0\\b_1,b_2\ge 0 \\ \epsilon=0,1}}u^{a_1+2a_2+\epsilon}v^{b_1+2b_2+\epsilon}\\
&=\frac{1}{1-u}\cdot \frac{1}{1-u^2}\cdot \frac{1}{1-v}\cdot \frac{1}{1-v^2}\cdot (1+uv).
\end{align*}
\end{proof}

\begin{lemma}\label{Hilb-sDAHA}
The Hilbert series of the bigraded algebra $ e\mathbb{H}_{t,q}^+e$  is given by 
$$\mathrm{Hilb}_{ e\mathbb{H}_{t,q}^+e}(u,v)=\frac{1}{1-u}\cdot \frac{1}{1-u^2}\cdot \frac{1}{1-v}\cdot \frac{1}{1-v^2}\cdot (1+uv).$$ 
\end{lemma}
\begin{proof}
Consider $e\mathbb{H}_{t,q}^+$, as a vector space over $\mathbb{C}$ or as a free module over $R_{t,q}$. By Proposition \ref{prop-AllRelnsDAHA} (2), a basis of $e\mathbb{H}_{t,q}^+$ is the set $\{e Y_1^{a_1} Y_2^{a_2}X_1^{b_1}X_2^{b_2} \ | \  a_i,b_i \in \mathbb{N}_0\}$. The group $S_2$ acts on $e\mathbb{H}_{t,q}^+$ by right multiplication, and the space of invariants is precisely the spherical subalgebra $e\mathbb{H}_{t,q}^+e$. 

This $S_2$ action depends polynomially on $t^{\pm1}$, does not depend on $q$, and it preserves the finite dimensional bigraded pieces $eH^+_{t,q}[N,M]$. The dimension of the space of invariants in every bigraded piece can be calculated using group characters as 
$$\dim (e\mathbb{H}_{t,q}^+[N,M]e)=  \frac{1}{2}\cdot (\mathrm{Tr}(\mathrm{id_{e\mathbb{H}_{t,q}^+[N,M]}})+\mathrm{Tr}(\mathrm{s_{e\mathbb{H}_{t,q}^+[N,M]}})).$$
This expression takes values in $\mathbb{N}_0$, does not depend on $q$, and the only term which depends on $t$ is $\mathrm{Tr}(\mathrm{s_{(e\mathbb{H}_{t,q}^+[N,M]}})$, which is in $\mathbb{C}[t^{\pm1}][(t^2+1)^{-1}]$. Thus, the above formula gives a polynomial function from $\{t\in \mathbb{C}^\times | t^2\ne 1 \}$ to $\mathbb{Z}$. The only such functions are constants. This proves that $\dim (e\mathbb{H}_{t,q}^+[N,M]e)$ does not depend on $t$ and $q$. 

We will calculate the Hilbert series of $e\mathbb{H}_{t,q}^+e$ by calculating it at $t=1,q=1$. The graded action of $S_2$ on $eH^+_{1,1}$ can be identified with the usual permutation action of $S_2$ on the space $\mathbb{C}[X_1,X_2, Y_1,Y_2]$. Denoting the trivial character of the group $S_2$ by  $\chi_+$ and the sign character by $\chi_-$, we see that the two dimensional  permutation representation spanned by $X_1,X_2 $ has graded character $(\chi_++\chi_-)u$. Consequently, the bigraded characters of $\mathbb{C}[X_{1},X_{2}]$ and $\mathbb{C}[X_{1},X_{2},Y_1,Y_2]$ can be computed as
\begin{align*}
\chi_{\mathbb{C}[X_{1},X_{2}]}(u,v)&=\frac{1}{1-v\chi_+}\cdot \frac{1}{1-v\chi_-}\cdot \chi_+\\
\chi_{\mathbb{C}[X_{1},X_{2}, Y_{1}, Y_{2}]}(u,v)&=\frac{1}{1-u\chi_+}\cdot \frac{1}{1-u\chi_-}\cdot \frac{1}{1-v\chi_+}\cdot \frac{1}{1-v\chi_-}\cdot \chi_+\\
&= \frac{\chi_+}{1-u} \cdot \frac{\chi_++u\chi_-}{1-u^2} \cdot \frac{\chi_+}{1-v} \cdot \frac{\chi_++v\chi_-}{1-v^2}\\
&= \frac{(1+uv)\chi_++(u+v)\chi_-}{(1-u)(1-u^2)(1-v)(1-v^2)}.
\end{align*}
Reading off the $\chi_+$ coefficient, we get that the bigraded Hilbert series of the spherical DAHA equals
\begin{align*}
\mathrm{Hilb}_{ e\mathbb{H}_{t,q}^+e}(u,v)&= \mathrm{Hilb}_{ e\mathbb{H}_{1,1}^+e}(u,v)\\
&=\frac{1}{1-u}\cdot \frac{1}{1-u^2}\cdot \frac{1}{1-v}\cdot \frac{1}{1-v^2}\cdot (1+uv).
\end{align*}
\end{proof}

We are now ready to prove the main Theorem of section \ref{review-DAHA}, Theorem \ref{thm-sDAHApresentation}, which describes the spherical DAHA by generators and relations and gives a PBW basis for it. 

\begin{proof}[Proof of Theorem \ref{thm-sDAHApresentation}]
Let us prove that the map $\Phi:\mathcal{A}\to e\mathbb{H}_{t,q}e$, defined on generators by \eqref{sDAHA-pres-iso}, is an isomorphism.

 It is a $\mathbb{Z}^2$ graded algebra homomorphism by Lemma \ref{Phi-is-hom}, and injective by Lemma \ref{Phi-inj}. Its restriction $\Phi|_{\mathcal{A}^+}:\mathcal{A}^+\to e\mathbb{H}_{t,q}^+e$ to the algebra $\mathcal{A}^+$ from Definition \ref{A+defn} is thus also a graded injective homomorphism. By Lemma \ref{Hilb-A} and Lemma \ref{Hilb-sDAHA}, the algebras  $\mathcal{A}^+$ and $e\mathbb{H}_{t,q}^+e$ have the same Hilbert series, so $\Phi|_{\mathcal{A}^+}$ is an isomorphism. By Lemma \ref{localiczationDAHA}, the homomorphism $\Phi$ is a localization of the isomorphism $\Phi|_{\mathcal{A}^+}$, and thus it is also an isomorphism. 

In particular, the image under $\Phi$ of the basis given in Lemma \ref{PBW-A} is a basis of the spherical DAHA $e\mathbb{H}_{t,q}e$. 
\end{proof}

\section{Proof of Theorem \ref{HamiltonPresentation}}\label{PBW-DqG}

This section contains the proof of Theorem \ref{HamiltonPresentation}, which gives the presentation of the quantum Hamiltonian reduction $$\cD_q(GL_2)\dS_{\!\!\! \mathcal{I}_t} U=\left(\mathcal{D}_q(GL_2)\Big/\cD_q(GL_2)\cdot\mu_q(\mathcal{I}_t)\right)^U\cong \mathcal{D}_q(GL_2)^U\Big/\cD_q(GL_2)^U\cdot \mu_q(Z_t)$$ by generators and relations, and a PBW type basis for it. To this end, we first give a presentation of $\mathcal{D}_q(GL_2)^U$, in Proposition \ref{PresentationInvariants}.

For most of this section we allow $q,t$ formal, or $q,t\in \mathbb{C}^\times$ arbitrary. Lemma \ref{lemma-Hilb-DqU} requires the additional assumption that $q$ is not a nontrivial root of $1$. The spirit of the proof is very similar to the proof of Theorem \ref{thm-sDAHApresentation} given in Section \ref{sec-DAHAPResentation}.

\subsection{A presentation of $\mathcal{D}_q(GL_2)^U$}

The main result of this subsection, and the only part of it which is used in the rest of the paper, is the following proposition.
\begin{proposition}\label{PresentationInvariants}

Assume that $q$ is not a non-trivial root of unity.

\begin{enumerate}
\item Let $\mathcal{B}$ be the algebra with generators:
$$c_1,c_2^{\pm 1},d_1,d_2^{\pm1},r,w$$
and relations:\\
\begin{minipage}[t]{0.4\textwidth}
\begin{align}
c_2c_1&=c_1c_2 \label{c1c2} \\
d_2d_1&=d_1d_2  \label{d1d2} \\
d_2c_2&=q^{-4}c_2d_2 \label{d2c2} \\
d_2c_1&=q^{-2}c_1d_2\label{d2c1} \\
d_1c_2&=q^{-2}c_2d_1 \label{d1c2} \\
d_1c_1&=c_1d_1+(q^{-2}-1)r \label{d1c1} 
\end{align}
\end{minipage}
\begin{minipage}[t]{0.55\textwidth}
\begin{align}
d_2r&= q^{-2}rd_2 \label{d2r}\\
rc_2&= q^{-2}c_2r \label{rc2} \\
d_1r&=q^{-2}rd_1+(1-q^{-2})q^{-2}c_1d_2  \label{d1r} \\
rc_1&= q^{-2}c_1r+(1-q^{-2}q^{-2}c_2d_1 \label{rc1} \\
r^2&=q^{-4}w +(q^{-4}+q^{-6})c_2d_2 -\label{r2reln}\\
& \quad -q^{-4}c_2d_1^2- q^{-4}c_1^2d_2+q^{-2}c_1rd_1 \notag 
\end{align}
\end{minipage}\\
\begin{minipage}[t]{0.3\textwidth}
\begin{align}
wr&=rw \label{wr} \\
wc_1&=q^{-2}c_1w \label{wc1} \\
wc_2&=q^{-4}c_2w\label{wc2} \\
wd_1&=q^{2}d_1w \label{wd1} \\
wd_2&=q^{4}d_2w. \label{wd2} 
\end{align}
\end{minipage}
\medskip

There is an isomorphism of algebras $$\overline{\Psi}:\mathcal{B}\to \left(\mathcal{D}_q(GL_2)\right)^U$$ given on the generators by
\begin{align}
\begin{aligned}
\overline{\Psi}(c_1)&= \tr_q(A), \quad & \quad \overline{\Psi}(c_2)&=  {\det}_q(A), \\
\overline{\Psi}(d_1)&= \tr_q(D), \quad & \quad \overline{\Psi}(d_2)&=  {\det}_q(D), \\
\overline{\Psi}(r) & = q^2 \tr_q(DA), \quad & \quad \overline{\Psi}(w)&=\mathrm{tr}_q(D\tilde{A}\tilde{D}A). 
\end{aligned}
\label{inv-pres-iso}
\end{align}

\item Let $\mathcal{B}^+$ be the subalgebra of $\mathcal{B}$ generated by $c_1,c_2,d_1,d_2,r,w$. The isomorphism $\overline{\Psi}$ restricts to the isomorphism of algebras $$\overline{\Psi}|_{\mathcal{B}^+}:\mathcal{B}^+\to \left(\mathcal{D}_q^+(GL_2)\right)^U.$$
\end{enumerate}
\end{proposition}

The algebras $\mathcal{B}$ and  $\mathcal{B}^+$ are $\mathbb{Z}^2$-graded by 
\begin{align}
\begin{aligned}\label{grading-B}
\deg(c_1)&=(1,0)  \qquad \deg(d_1)=(0,1) \qquad \deg(w)=(2,2)\\
\deg(c_2)&=(2,0)\qquad \deg(d_2)=(0,2) \qquad \deg(r)=(1,1).
\end{aligned}
\end{align}
The multiplicative set generated by $c_2d_2$ and $w$ is $q$-central, and therefore Ore in $\mathcal{B}^+$, and $\mathcal{B}$ is a noncommutative localization of $\mathcal{B}^+$ by this set.

Let us first show that the map $\overline{\Psi}$ is indeed a homomorphism. 

\begin{lemma}\label{lemma-psi'-is-hom}
The elements 
$$\tr_q(A), {\det}_q(A), \tr_q(D),  {\det}_q(D), q^2 \tr_q(DA), \mathrm{tr}_q(D\tilde{A}\tilde{D}A),$$
which appear on the right hand sides of formulas \eqref{inv-pres-iso}, are $U$-invariant and satisfy defining relations \eqref{c1c2}-\eqref{wd2} of the algebra $\mathcal{B}$. Therefore, \eqref{inv-pres-iso} defines a graded algebra homomorphism $\overline{\Psi}: \mathcal{B}\to \left(\mathcal{D}_q(GL_2)\right)^U.$
\end{lemma}
\begin{proof}

The maps $L\mapsto A$, $L\mapsto D$, $L\mapsto DA$ and $L\mapsto \mu_q(L)$ are all homomorphisms of $U$-modules, so the $U$-invariance of these elements follows from Lemma \ref{trdetarecenter}.

Seeing that these elements satisfy relations \eqref{c1c2}-\eqref{wd2} is also a direct computation, using the explicit form of defining relations of $\cD_q(GL_2)$ given by \eqref{relns-aa}, \eqref{relns-dd} and \eqref{relns-da}. For instance, \eqref{grading-Dq-inner} immediately implies \eqref{c1c2}, \eqref{d1d2},  \eqref{d2c2}, \eqref{d2c1}, \eqref{d1c2}, \eqref{d2r}, \eqref{rc2}, \eqref{wc2}, \eqref{wd2}. 

For illustration, let us also prove \eqref{d1c1}. We claim that 
$$\overline{\Psi}(d_1)\overline{\Psi}(c_1)-\overline{\Psi}(c_1)\overline{\Psi}(d_1)=(q^{-2}-1)\overline{\Psi}(r).$$
The left hand side is equal to 
\begin{align*}
&\overline{\Psi}(d_1)\overline{\Psi}(c_1)-\overline{\Psi}(c_1)\overline{\Psi}(d_1) =\\
&\quad  \stackrel{\phantom{\eqref{relns-da}}}{=} \tr_q(D) \tr_q(A) -\tr_q(A) \tr_q(D) \\
& \quad \stackrel{\phantom{\eqref{relns-da}}}{=} (\partial^1_1+q^{-2}\partial^2_2)\, (a^1_1+q^{-2}a^2_2)- (a^1_1+q^{-2}a^2_2)\, (\partial^1_1+q^{-2}\partial^2_2)\\
& \quad \stackrel{\eqref{relns-da}}{=} (q^{-2}-1)\left(a^1_1\partial^1_1+(q^{-2}-1)a^1_1\partial^2_2+q^2a^1_2\partial^2_1+a^2_1\partial^1_2+(q^{-2}-1)a^2_2\partial^1_1+ \right.\\
& \quad\quad \quad \quad +\left. (1-q^{-2}+q^{-4})a^2_2\partial^2_2\right).
\end{align*}
 On the other hand, 
\begin{align}
\overline{\Psi}(r) &\stackrel{\phantom{\eqref{relns-da}}}{=}  q^2\tr_q(DA) \notag\\
&\stackrel{\phantom{\eqref{relns-da}}}{=} q^2(\partial^1_1 a^1_1+\partial^1_2 a^2_1)+(\partial^2_1 a^1_2+\partial^2_2 a^2_2)  \notag \\
&\stackrel{\eqref{relns-da}}{=}a^1_1\partial^1_1+(q^{-2}-1)a^1_1\partial^2_2+q^2a^1_2\partial^2_1+q^2a^1_2\partial^2_1+a^2_1\partial^1_2+ \notag \\
 & \quad\quad   +(q^{-2}-1)a^2_2\partial^1_1+ (1-q^{-2}+q^{-4})a^2_2\partial^2_2.\label{Psi'-r-explicit}
\end{align}
Comparing these two expressions proves \eqref{d1c1}. Relations \eqref{d1r}, \eqref{rc1}, \eqref{r2reln}, \eqref{wr}, \eqref{wc1} and \eqref{wd1} are proved by similar direct computations. 
\end{proof}

We now describe a PBW basis of the algebra $\mathcal{B}$ and the Hilbert series of $\mathcal{B}^+$.


\begin{lemma}\label{PBWauxB}
\begin{enumerate}
\item The set $$\{c_1^{a_1}c_2^{a_2} r^\epsilon d_1^{b_1}d_2^{b_2} w^c \,|\, a_1,b_1,c\in \mathbb{N}_0, a_2,b_2 \in \mathbb{Z},\epsilon \in \{0,1\}  \}$$
is a basis of $\mathcal{B}$.
\item The set $$\{c_1^{a_1}c_2^{a_2} r^\epsilon d_1^{b_1}d_2^{b_2} w^c \,|\, a_1,b_1,a_2,b_2,c\in \mathbb{N}_0,\epsilon \in \{0,1\}  \}$$
is a basis of $\mathcal{B^+}$.
\item The Hilbert series of the algebra $\mathcal{B^+}$ 
is given by
 $$\mathrm{Hilb}_{ \mathcal{B}^+}(u,v)=\frac{1}{(1-u)(1-v)(1-u^2)(1-v^2)(1-uv)}.
$$
\end{enumerate}
\end{lemma}
\begin{proof}
\begin{enumerate}
\item The first claim follows from the second by localization at $c_2$ and $d_2$.
\item The second claim is an application of the Diamond Lemma from \cite{B}, similar to the proof of Lemma \ref{PBW-A}. Relations \eqref{c1c2}-\eqref{r2reln} are straightening relations, describing how to reorder monomials. In order to apply the Diamond Lemma, it remains to establish the straightening diamonds for the following 26 monomials: 
$$wd_2d_1, \,\, wd_2r, \,\, wd_2c_2, \,\, wd_2c_1, \,\, wd_1r, \,\, wd_1c_2, \,\, wd_1c_1, \,\, wrc_2, \,\, wrc_1, \,\,  wc_2c_1,  $$
$$d_2d_1r, \,\, d_2d_1c_2, \,\, d_2d_1c_1, \,\, d_2rc_2, \,\, d_2rc_1, \,\, d_2c_2c_1, \,\, d_1rc_2, \,\, r^2c_2, \,\, d_1c_2c_1, \,\,  rc_2c_1,  $$
$$wr^2, \,\, d_2r^2, \,\, d_1r^2, \,\, d_1rc_1, \,\, r^2c_1, \,\, r^3.$$

This is a direct calculation, very similar to that in the proof of Lemma \ref{PBW-A}.  As there, all but the final four monomials follow immediately from the fact that $c_2,d_2$ and $w$ $q$-commute with all generators.  For illustration, we include a proof that the straightening diamond for $r^2c_1$ holds. Straightening $r^2$ first, we find:
\begin{align*}
(r^2)c_1&\stackrel{\eqref{r2reln}}{=} (q^{-2}c_1rd_1- q^{-4}c_1^2d_2-q^{-4}c_2d_1^2+q^{-4}w +(q^{-4}+q^{-6})c_2d_2)\, c_1\\
&\stackrel{\substack{\eqref{d1c1} \\ \eqref{wc1} \\ \eqref{d2c1}}}{=} q^{-2}c_1rc_1d_1+(q^{-4}-q^{-2})c_1r^2- q^{-6}c_1^3d_2-q^{-4}c_2d_1c_1d_1-\\
& \quad  \quad  \quad -(q^{-6}-q^{-4})c_2d_1r+q^{-6}c_1w +(q^{-6}+q^{-8})c_2c_1d_2\\
&\stackrel{\substack{\eqref{rc1}\\ \eqref{r2reln} \\  \eqref{d1c1}\\ \eqref{d1r}}}{=}q^{-4}c_1^2rd_1+(q^{-4}-q^{-6})c_1c_2d_1^2
+(q^{-6}-q^{-4})c_1^2rd_1-\\
&\quad  \quad  \quad -(q^{-8}-q^{-6})c_1^3d_2-(q^{-8}-q^{-6})c_1c_2d_1^2+(q^{-8}-q^{-6})c_1w+\\
&\quad  \quad  \quad+(q^{-4}-q^{-2})(q^{-4}+q^{-6})c_1c_2d_2- q^{-6}c_1^3d_2-q^{-4}c_1c_2d_1^2-\\
& \quad  \quad  \quad -(q^{-6}-q^{-4}) c_2rd_1-(q^{-8}-q^{-6}) c_2rd_1  - \\
& \quad  \quad  \quad -(q^{-6}-q^{-4})(q^{-2}-q^{-4})c_1c_2d_2+q^{-6}c_1w +(q^{-6}+q^{-8})c_1c_2d_2\\
&\stackrel{\phantom{\eqref{r2reln}}}{=}q^{-6}c_1^2rd_1-q^{-8}c_1c_2d_1^2-q^{-8}c_1^3d_2+q^{-8}c_1w+(q^{-4}-q^{-8}) c_2rd_1 \\
&\quad  \quad  \quad+(2q^{-10}-q^{-8}+q^{-6})c_1c_2d_2.
\end{align*}
Straightening $rc_1$ first, we find:
\begin{align*}
r(rc_1)&\stackrel{\eqref{rc1}}{=} q^{-2}rc_1r+(q^{-2}-q^{-4})rc_2d_1\\
&\stackrel{\substack{\eqref{rc1} \\ \eqref{rc2}}}{=} q^{-4}c_1r^2+(q^{-4}-q^{-6})c_2d_1r +(q^{-4}-q^{-6})c_2rd_1\\
&\stackrel{\substack{ \eqref{r2reln} \\ \eqref{d1r}}}{=} q^{-6}c_1^2rd_1- q^{-8}c_1^3d_2-q^{-8}c_1c_2d_1^2+q^{-8}c_1w+(q^{-8}+q^{-10})c_1c_2d_2+\\
& \quad  \quad  \quad +(q^{-6}-q^{-8})c_2rd_1+(q^{-4}-q^{-6})(q^{-2}-q^{-4})c_1c_2d_2+(q^{-4}-q^{-6})c_2rd_1\\
&\stackrel{\substack{\eqref{r2reln}\\ \eqref{d1r}}}{=} q^{-6}c_1^2rd_1- q^{-8}c_1^3d_2-q^{-8}c_1c_2d_1^2+q^{-8}c_1w+(q^{-4}-q^{-8})c_2rd_1+\\
& \quad  \quad  \quad + (2q^{-10}-q^{-8}+q^{-6})c_1c_2d_2.
\end{align*}
By inspection, these two expressions agree. Similar computations show that straightening diamonds hold for the remaining expression; hence, by the Diamond Lemma, the given set of straightened monomials is indeed a PBW basis for $\mathcal{B}^+$.

\item Using 2), we calculate the Hilbert series of $\mathcal{B}^+$ as 
\begin{align*}
\mathrm{Hilb}_{ \mathcal{B}^+}(u,v)&=\sum_{\substack{a_1,b_1,a_2,b_2,c\ge 0 \\ \epsilon \in \{0,1\}}} u^{a_1+2a_2+\epsilon+2c}\, v^{b_1+2b_2+\epsilon+2c}\\
&=\frac{1+uv}{(1-u)(1-u^2)(1-v)(1-v^2)(1-u^2v^2)}\\
&=\frac{1}{(1-u)(1-v)(1-u^2)(1-v^2)(1-uv)}.
\end{align*}
\end{enumerate}
\end{proof}

Next, we show that $\overline{\Psi}|_{\mathcal{B}^+}$ is injective. 
\begin{lemma}\label{Dq-lin-ind}
The image of the set 
\begin{align}
\{c_1^{a_1}c_2^{a_2} r^\epsilon d_1^{b_1}d_2^{b_2} w^c \,|\, a_1,b_1,c\in \mathbb{N}_0,\, ,a_2,b_2 \in \mathbb{Z},\, \epsilon \in \{0,1\}  \} \label{basis-Dq-inv-try1}
\end{align}
under the map $\overline{\Psi}$ is linearly independent in $\left(\mathcal{D}_{q}(GL_2)\right)^{U}$. The map $\overline{\Psi}:\mathcal{B}\to \left(\mathcal{D}_{q}(GL_2)\right)^{U}$ is injective. 
\end{lemma}
\begin{proof}
Let us consider the set
\begin{align}
\{c_1^{a_1}c_2^{a_2} r^d d_1^{b_1}d_2^{b_2}  \,|\, a_1,b_1,d\in \mathbb{N}_0,a_2,b_2 \in \mathbb{Z}\} \label{replacement-basis}
\end{align}
with a partial order on it given by the degree $d$ of $r$. Looking at the defining relations \eqref{c1c2}-\eqref{r2reln} of the algebra $\mathcal{B}$, we see that 
$$c_1^{a_1}c_2^{a_2} r^\epsilon d_1^{b_1}d_2^{b_2} w^c=q^{-2b_1-4b_2+4c}c_1^{a_1}c_2^{a_2} r^{2c+\epsilon} d_1^{b_1}d_2^{b_2}+\textrm{ lower order terms}.$$
In other words, the set \eqref{replacement-basis} is also a basis of $\mathcal{B}$, and the change of basis matrix between \eqref{basis-Dq-inv-try1} and \eqref{replacement-basis}  (in some order) is upper triangular with powers of $q$ on the diagonal. So, it is enough to prove that the image under $\overline{\Psi}$ of the set \eqref{replacement-basis} is linearly independent in $\left(\mathcal{D}_{q}(GL_2)\right)^{U}$.

Suppose, for the sake of contradiction, that there is a nontrivial linear combination of the images under $\overline{\Psi}$ of some elements of \eqref{replacement-basis}.
\begin{align}
\sum_{\substack{a_1,b_1,d\in \mathbb{N}_0 \\ a_2,b_2 \in \mathbb{Z} }} \alpha_{a_1,a_2,d, b_1,b_2} \, \overline{\Psi}(c_1^{a_1}c_2^{a_2}  r^d d_1^{b_1}d_2^{b_2})=0.
 \label{Dq-lin-ind-eq1}
\end{align}
By this assumption, the set
$$E=\{ {\bf e}=(a_1,a_2,d,b_1,b_2) | a_1,b_1,d\in \mathbb{N}_0, a_2,b_2 \in \mathbb{Z}, \alpha_{a_1,a_2,d, b_1,b_2} \ne 0 \}$$
is not empty. 
After multiplying \eqref{Dq-lin-ind-eq1} on the left by some power of $\overline{\Psi}(c_2)$ and on the right by some power of $\overline{\Psi}(d_2)$, we may assume that 
$$E=E\cap \mathbb{N}_0^5,$$ 
meaning that all exponents $a_2,b_2$ are nonnegative integers.  
Let $p_1,\ldots, p_5:E\to \mathbb{N}_0$ be projections from $E$ to the coordinates of $E$. 
We have 
\begin{align}
\sum_{{\bf e}\in E} \alpha_{{\bf e}} \, \mathrm{tr}_q(A)^{a_{1}}\, \mathrm{det}_q(A)^{a_{2}} \, (q^2\mathrm{tr}_q(DA))^{d}\mathrm{tr}_q(D)^{b_{1}}\, \mathrm{det}_q(D)^{b_{2}}  =0
 \label{Dq-lin-ind-eq2}
\end{align}

Let
$$N=\max \{ (p_2+p_3)({\bf e}) | {\bf e}\in E\}.$$
The set $(p_2+p_3)^{-1}(N)\subseteq E$ is nonempty. 
Putting \eqref{Dq-lin-ind-eq2} in PBW order using \eqref{deftrdet}, \eqref{Psi'-r-explicit} and \eqref{relns-aa}, we read off that the leading nonzero term of $a^1_2$ 
 is equal to 
 \begin{align*}
\sum_{{\bf e}\in (p_2+p_3)^{-1}(N)} \alpha_{{\bf e}} \, (a^1_1+q^{-2}a^2_2)^{a_{1}}(-q^2)^N (a^1_2)^N(a^2_1)^{a_2}(\partial^2_1)^d \mathrm{tr}_q(D)^{b_{1}}\, \mathrm{det}_q(D)^{b_{2}}  =0.
\end{align*}
This can be rewritten as 
\begin{align}
\sum_{{\bf e}\in (p_2+p_3)^{-1}(N)}\sum_{i=0}^{a_1} \alpha_{{\bf e}} {a_1 \choose i} q^{-2i+2di}  \, (a^1_1)^{a_1-i}(a^1_2)^N (a^2_1)^{a_2}  (a^2_2)^{i}(\partial^2_1)^d \mathrm{tr}_q(D)^{b_{1}}\, \mathrm{det}_q(D)^{b_{2}}  =0
\label{Dq-lin-ind-eq2.5} \end{align}

Fix any $(a_1,a_2,d,b_1,b_2) \in (p_2+p_3)^{-1}(N)$. Then the set 
$$p_1^{-1}(a_1)\cap p_2^{-1}(a_2) \cap p_3^{-1}(d)$$
is not empty. 
Using the PBW theorem for $\cD_{q}(GL_2)$, (Proposition \ref{Dq-PBW-prop}) and equation \eqref{Dq-lin-ind-eq2.5}, we see that 
\begin{align}
\sum_{{\bf e}\in p_1^{-1}(a_1)\cap p_2^{-1}(a_2) \cap p_3^{-1}(d)} \alpha_{{\bf e}} \, \mathrm{tr}_q(D)^{p_4({\bf e})}\, \mathrm{det}_q(D)^{p_5({\bf e})}=0.
\label{Dq-lin-ind-eq3}
\end{align}
Let $$M=\max \{p_5({\bf e}) | {\bf e}\in p_1^{-1}(a_1)\cap p_2^{-1}(a_2) \cap p_3^{-1}(d) \}.$$
The set 
$$p_1^{-1}(a_1)\cap p_2^{-1}(a_2) \cap p_3^{-1}(d)\cap p_5^{-1}(M)$$
is nonempty.
The leading $\partial^1_2$ term in \eqref{Dq-lin-ind-eq3} is then  
\begin{align}
\sum_{{\bf e}\in p_1^{-1}(a_1)\cap p_2^{-1}(a_2) \cap p_3^{-1}(d)\cap p_5^{-1}(M)} \alpha_{{\bf e}} \, \mathrm{tr}_q(D)^{p_4({\bf e})}\, (\partial^1_2)^{M}(\partial^2_1)^{M}=0.
\end{align}
From this it follows that 
$$\alpha_{{\bf e}}=0 \quad \textrm{for all } \quad {\bf e}\in p_1^{-1}(a_1)\cap p_2^{-1}(a_2) \cap p_3^{-1}(d)\cap p_5^{-1}(M)\subseteq E,$$
which contradicts the definition of the set $E$. 

So, the assumption that there exists a nontrivial linear combination of the images under $\overline{\Psi}$ of some elements of \eqref{replacement-basis} is wrong. This proves the lemma.

\end{proof}

We have now shown that the map $\overline{\Psi}$ from the statement of Proposition \ref{PresentationInvariants} is a graded injective homomorphism of algebras. To see that it is an isomorphism, we employ a dimension argument. The algebras $\mathcal{B}$ and $\mathcal{D}_q(GL_2)$ have infinite dimensional bigraded pieces, so we first deal with their subalgebras $\mathcal{B}^+$ and $\mathcal{D}_q^+(GL_2)$. 

\begin{lemma}\label{lemma-Hilb-DqU}
Assume that $q$ is not a nontrivial root of unity.  Then the Hilbert series of the algebra $\mathcal{D}_q^+(GL_2)^U$ is 
$$\mathrm{Hilb}_{ \mathcal{D}_q^+(GL_2)^U}(u,v)=\frac{1}{(1-u)(1-v)(1-u^2)(1-v^2)(1-uv)}.$$
\end{lemma}
\begin{proof}
As a bigraded $U$-module, $\mathcal{D}_q^+(GL_2)$ is isomorphic to $\cO^+_q(GL_2)\otimes \cO^+_q(GL_2)$. In order to find the invariants, we will decompose the representation $\cO^+_q(GL_2)$ of $U$ into irreducible direct summands. The multiplicities of irreducible modules in this decomposition is the same for a variable $q$, for all $q \in \mathbb{C}^\times$ which are not roots of unity, and for $q=1$. (For $q\ne 1$ such that $q^n=1$, there are genuinely more invariants.) In particular, 
$$\mathrm{Hilb}_{ \mathcal{D}_q^+(GL_2)^U}(u,v)=\mathrm{Hilb}_{ \mathcal{D}_{1}^+(GL_2)^{U(\mathfrak{gl}_2)}}(u,v).$$
We will calculate this Hilbert series at $q=1$. 

At $q=1$, the algebra $U=U(\mathfrak{gl}_2)$ is the universal enveloping algebra of $\mathfrak{gl}_2$. Let $V_n$ denote the irreducible representation of $U(\mathfrak{gl}_2)$ which factors through the quotient $U(\mathfrak{gl}_2)\to U(\mathfrak{sl}_2)$ and whose highest weight is $n$. We have $\dim V_n=n+1$. Then $$\mathrm{span}\{a^i_j \, | \,  i,j\in \{1,2\} \}\cong V_0\oplus V_2$$ as $U(\mathfrak{gl}_2)$ modules.
For any $U(\mathfrak{gl}_2)$ module $V$, let $S(V)$ denote the symmetric algebra on $V$ with the obvious grading. There is an isomorphism of graded $U(\mathfrak{gl}_2)$ modules \begin{align}\label{Groth-SV0+SV2}
\cO^+_{1}(GL_2)=\mathbb{C}[a^i_j \, | \,  i,j\in \{1,2\} ]\cong S(V_0\oplus V_2)\cong S(V_0)\otimes S(V_2).
\end{align}

Let us record the decomposition of graded pieces into irreducible direct summands as a one-variable Grothendieck group expression. In that language, 
\begin{align}\label{Groth-SV0}
[S(V_0)]=\sum_{k\ge 0} [V_0] u^k.
\end{align}
Using weights, we decompose the symmetric algebra on $V_2$ as a direct sum of irreducible representations, and similarly get that
\begin{align}\label{Groth-SV2}
[S(V_2)]=\sum_{n\ge 0} [\oplus_{i\ge 0}V_{2n-4i} ] u^n=\sum_{n\ge 0}\sum_{i\ge 0} [V_{2n-4i} ] u^n.
\end{align}
Combining \eqref{Groth-SV0+SV2}, \eqref{Groth-SV0} and \eqref{Groth-SV2}, we get 
\begin{align}
[\cO^+_{1}(GL_2)]&=\left( \sum_{k\ge 0} [V_0] u^k \right) \left( \sum_{n\ge 0}\sum_{i\ge 0} [V_{2n-4i} ] u^n\right) = \sum_{k\ge 0}\sum_{n\ge 0}\sum_{i\ge 0} [V_{2n-4i}] u^{n+k}. \label{Groth-O_1}
\end{align}

We are interested in the multiplicity of $[V_0]$ in $\cO^+_{1}(GL_2)\otimes \cO^+_{1}(GL_2)$. For that purpose, recall that 
$$V_m\otimes V_n\cong V_{m+n}\oplus V_{m+n-2}\oplus\cdots \oplus V_{|m-n|}.$$
The trivial representation $V_0$ appears as a summand in the decomposition of $V_m\otimes V_n$ if and only if $m=n$, and in that case it appears with multiplicity $1$. So, 
\begin{align*}
[\cO^+_{1}(GL_2)\otimes \cO^+_{1}(GL_2)]&=\sum_{k\ge 0}\sum_{n\ge 0}\sum_{i\ge 0} [V_{2n-4i}] u^{n+k}\cdot \sum_{l\ge 0}\sum_{m\ge 0}\sum_{j\ge 0} [V_{2m-4j}] v^{m+l}
\end{align*}
and 
\begin{align*}
[(\cO^+_{1}(GL_2)\otimes \cO^+_{1}(GL_2))^{U(\mathfrak{gl}_2)}]&=\sum_{\substack{k\ge 0 \\ n\ge 0 \\ 0\le i \le n/2}}\sum_{\substack{l\ge 0 \\ m\ge 0 \\ 0\le j \le m/2}} [(V_{2n-4i} \otimes V_{2m-4j})^{U(\mathfrak{gl}_2)}]  u^{n+k}v^{m+l}\\
&=[V_0] \sum_{\substack{k\ge 0 \\ n\ge 0 \\ 0\le i \le n/2}}\sum_{\substack{l\ge 0 \\ m\ge 0 \\ 0\le j \le m/2}}\delta_{2n-4i=2m-4j}  u^{n+k}v^{m+l}.
\end{align*}
The Hilbert series of the space of invariants is thus
\begin{align*}
\mathrm{Hilb}_{ \mathcal{D}_{1}^+(GL_2)^{U(\mathfrak{gl}_2)}}(u,v)&=\sum_{\substack{k\ge 0 \\ n\ge 0 \\ 0\le i \le n/2}}\sum_{\substack{l\ge 0 \\ m\ge 0 \\ 0\le j \le m/2}}\delta_{2n-4i=2m-4j}  u^{n+k}v^{m+l}\\
&=\frac{1}{(1-u)(1-v)}\sum_{ i\ge 0}\sum_{j\ge 0}\sum_{ n\ge 2i} \sum_{ m\ge 2j}\delta_{2n-4i=2m-4j}  u^{n}v^{m}\\
&=\frac{1}{(1-u)(1-v)}\sum_{ i\ge 0}\sum_{j\ge 0}\sum_{ n\ge 2i}  u^{n}v^{n-2i+2j}\\
&=\frac{1}{(1-u)(1-v)}\sum_{ i\ge 0}\sum_{j\ge 0}\sum_{n\ge 0}  u^{n+2i}v^{n+2j}\\
&=\frac{1}{(1-u)(1-v)(1-u^2)(1-v^2)(1-uv)}.
\end{align*}
\end{proof}

We are now ready to prove the main statement of this subsection.

 \begin{proof}[Proof of Proposition \ref{PresentationInvariants}] 

Assume that $q,t \neq 0,$ and that $q$ is not a nontrivial root of unity. 

The map $\overline{\Psi}$, defined on generators of $\mathcal{B}$ in \eqref{inv-pres-iso} extends uniquely to a graded homomorphism of algebras $\overline{\Psi}:\mathcal{B}\to (\cD_q(GL_2))^U$ by Lemma \ref{lemma-psi'-is-hom}. By Lemma \ref{PBWauxB} and Lemma \ref{Dq-lin-ind}, the image of a basis in $\mathcal{B}$ is linearly independent in $(\cD_q(GL_2))^U$, so the map $\overline{\Psi}$ is injective. The restriction of $\overline{\Psi}$ to $\mathcal{B}^+$ is thus also a graded injective homomorphism of algebras $\overline{\Psi}|_{\mathcal{B}^+}:{\mathcal{B}^+} \to (\cD_{q}^+(GL_2))^U$. The Hilbert series of ${\mathcal{B}^+}$ and $(\cD_{q}^+(GL_2))^U$ given in Lemmas \ref{PBWauxB} and \ref{lemma-Hilb-DqU} coincide, so $\overline{\Psi}|_{\mathcal{B}^+}\to (\cD_{q}^+(GL_2))^U$ is an isomorphism. The algebra $\mathcal{B}$ is a localization of $\mathcal{B}^+$ by the Ore set generated by $c_2d_2$, the algebra $ (\cD_q(GL_2))^U$ is a localization of $(\cD_{q}^+(GL_2))^U$ by the Ore set generated by ${\det}_q(A){\det}_q(D)=\overline{\Psi}(c_2d_2)$, and the map $\overline{\Psi}$ is a localization of the map $\overline{\Psi}|_{\mathcal{B}^+}$. So, $\overline{\Psi}:\mathcal{B}\to (\cD_q(GL_2))^U$ is an isomorphism of algebras.

\end{proof}

\subsection{A presentation of $\cD_q(GL_2)\dS_{\!\!\! \mathcal{I}_t} U$}

We now combine Lemma \ref{ReduceToInvariants} and Proposition \ref{PresentationInvariants}, to prove Theorem \ref{HamiltonPresentation}.

\begin{proof}[Proof of Theorem \ref{HamiltonPresentation}] 

By Lemma \ref{ReduceToInvariants} and Proposition \ref{PresentationInvariants}, we have 
$$\mathcal{D}_q(GL_2)\dS_{\!\!\!\mathcal{I}_t} U\cong \mathcal{D}_q(GL_2)^U\Big/ \mathcal{D}_q(GL_2)^U \cdot \mu_q(Z_t)  \cong  \mathcal{B}\Big/ \mathcal{B}\cdot\overline{\Psi}^{-1}\left(\mu_q(Z_t) \right)$$

Let us first calculate the generator $\overline{\Psi}^{-1}\left(\mu_q(Z_t) \right)$ of this principal $\mathcal{B}$ ideal.
\begin{align*}
\overline{\Psi}^{-1}\left(\mu_q(Z_t) \right) 
&\stackrel{Def \ref{def-Zt}}{=} \overline{\Psi}^{-1}(\mu_q(\mathrm{tr}_q(L-q^4X_t)))\\
&\stackrel{Def \ref{def-muq}}{=}\overline{\Psi}^{-1}(\mathrm{tr}_q(DA^{-1}D^{-1}A-q^4X_t))\\
&\stackrel{\eqref{def-tilde}}{=}\overline{\Psi}^{-1}(\mathrm{tr}_q(D\, {\det}_q(A)^{-1}\tilde{A}\,  {\det}_q(D)^{-1}\tilde{D}A-q^4X_t))\\
&\stackrel{\phantom{\eqref{det-is-q-central}}}{=}\overline{\Psi}^{-1}( {\det}_q(A)^{-1} {\det}_q(D)^{-1} \mathrm{tr}_q(D\tilde{A}\tilde{D}A)-q^4 \mathrm{tr}_q(X_t))\\
&\stackrel{\eqref{inv-pres-iso}}{=}c_2^{-1}d_2^{-1}w-q^4(t^{-2} + q^{-2}t^2)\\
&\stackrel{\phantom{\eqref{inv-pres-iso}}}{=}c_2^{-1}d_2^{-1}\left (w-q^4(t^{-2} + q^{-2}t^2)d_2c_2\right)\\
&\stackrel{\eqref{d2c2}}{=}c_2^{-1}d_2^{-1}\left (w-(t^{-2} + q^{-2}t^2)c_2d_2\right).
\end{align*}

Thus $\mathcal{B}\cdot\overline{\Psi}^{-1}\left(\mu_q(Z_t) \right)$ is the principal left ideal in $\mathcal{B}$ generated by the element $c_2^{-1}d_2^{-1} (w-(t^{-2} + q^{-2}t^2)c_2d_2),$ or equivalently, by the element 
$w-(t^{-2} + q^{-2}t^2)c_2d_2$. A presentation of the quotient, $$\mathcal{B}\Big/\mathcal{B}\cdot(w-(t^{-2} + q^{-2}t^2)c_2d_2)$$ can be deduced from the presentation of the algebra $\mathcal{B}$ given in Proposition \ref{PresentationInvariants}, as follows: 
\begin{itemize}
\item Generators are: $c_1,c_2^{\pm 1},d_1,d_2^{\pm1},r$.
\item Relations are \eqref{c1c2}-\eqref{rc1}, and an extra relation obtained from \eqref{r2reln} as 
\begin{align}
r^2&=q^{-4}w +(q^{-4}+q^{-6})c_2d_2 -q^{-4}c_2d_1^2- q^{-4}c_1^2d_2+q^{-2}c_1rd_1\notag\\
&=q^{-4}(t^{-2} + q^{-2}t^2)c_2d_2 +(q^{-4}+q^{-6})c_2d_2 -q^{-4}c_2d_1^2- q^{-4}c_1^2d_2+q^{-2}c_1rd_1\notag\\
&=q^{-4}(1+t^2)(q^{-2}+t^{-2})c_2d_2-q^{-4}c_2d_1^2- q^{-4}c_1^2d_2+q^{-2}c_1rd_1 .\label{r2relnInQuot}
\end{align}
\end{itemize}

This proves that $\cD_q(GL_2)\dS_{\!\!\!\mathcal{I}_t}U$ is isomorphic to the algebra $\mathcal{B}\Big/\mathcal{B}\cdot(w-(t^{-2} + q^{-2}t^2)c_2d_2)$, whose presentation is stated in Theorem \ref{HamiltonPresentation}, with the isomorphism $\Psi$ induced by $\overline{\Psi}$.

\end{proof}

\end{document}